\documentclass{amsart}

\usepackage{amsmath,amssymb,dsfont}

\numberwithin{equation}{section}

\newcommand{\R}{\mathbb{R}}
\newcommand{\C}{\mathbb{C}}
\newcommand{\Z}{\mathbb{Z}}

\renewcommand{\to}{\rightarrow}
\newcommand{\weakto}{\rightharpoonup}

\newcommand{\Ldif}{\mathcal{L}}

\renewcommand{\leq}{\leqslant}
\renewcommand{\geq}{\geqslant}

\newcommand{\eps}{\varepsilon}

\newcommand\1{{\ensuremath {\mathds 1} }}

\newtheorem{thm}{Theorem}[section]
\newtheorem{lemma}{Lemma}[section]
\newtheorem{defin}{Definition}[section]
\newtheorem*{defin*}{Definition}

\newtheorem{prop}{Proposition}[section]
\newtheorem{remark}{Remark}[section]

\numberwithin{equation}{section}
\begin{document}

\title[Continuum limit for discrete NLS with long-range lattice interactions]{On the continuum limit for discrete NLS with long-range lattice interactions}
\author{Kay Kirkpatrick}
\address{University of Illinois at Urbana-Champaign, Department of Mathematics, 1409 W. Green Street, Urbana, IL 61801}
%of Mathematical Sciences, New York University, and Centre De Recherche en MathŽmatiques de la D\'ecision, Universit\'e Paris IX Dauphine, Place du Mar\'echal de Lattre de TASSIGNY, F-75775 Paris Cedex 16 (France) }
\email{kkirkpat@illinois.edu}
\thanks{K.K. was partially supported by NSF grants DMS-0703618, DMS-1106770 and OISE-0730136}
\author{Enno Lenzmann}
\address{Mathematisches Institut, Universit\"at Basel, Rheinsprung 21, CH-4051 Basel, Switzerland.}
\email{enno.lenzmann@unibas.ch}
\thanks{E.L. acknowledges support by a Steno fellowship from the Danish Research Council}
\author{Gigliola Staffilani}
\address{Massachusetts Institute of Technology, Room 2-246, 77 Massachusetts Avenue, Cambridge, MA 02138}
\email{gigliola@math.mit.edu}
\thanks{G.S. was partially supported by NSF grant DMS-1068815 }

\maketitle
\begin{abstract}
We consider a general class of discrete nonlinear Schr\"odinger equations (DNLS) on the lattice $h \mathbb{Z}$ with mesh size $h >0$. In the continuum limit when $h \to 0$, we prove that the limiting dynamics are given by a nonlinear Schr\"odinger equation (NLS) on $\R$ with the fractional Laplacian $(-\Delta)^\alpha$ as dispersive symbol. In particular, we obtain that fractional powers $\frac 1 2 < \alpha < 1$ arise from long-range lattice interactions when passing to the continuum limit, whereas the NLS with the usual Laplacian $-\Delta$ describes the dispersion in the continuum limit for short-range or quick-decaying interactions (e.\,g., nearest-neighbor interactions).

Our results rigorously justify certain NLS model equations with fractional Laplacians proposed in the physics literature. Moreover, the arguments given in our paper can be also applied to discuss the continuum limit for other lattice systems with long-range interactions.
\end{abstract}

\section{Introduction}

%Partial differential equations (PDEs) with fractional Laplacians arise in a variety of physical models, e.\,g., in phase transitions, long-jump random walks, water waves, astrophysics, and quantum biophysics. See, e.\,g., \cite{SV,LV}.  

In the present paper, we show how PDEs with fractional Laplacians $(-\Delta)^\alpha$ can be rigorously derived as the continuum limit of certain discrete physical systems with long-range lattice interactions. In fact, this theme is of interest in the recent physics literature, where only formal arguments are presented; see, e.\,g., \cite{Las, Tar,TarZa,GMCR1,GMCR2,MCGJR}. In this work  our  rigorous arguments are for the derivation of these nonlocal continuum dynamics in the case of fractional NLS-type equations, for the sake of simplicity. But in fact, the arguments developed below will have applications to continuum limits for other types of discrete evolution equations with long-range interactions.   

As a specific physical example, we take a family of models for charge transport in biopolymers like the DNA; see, e.\,g., \cite{GMCR1, GMCR2, MCGJR}. Here, the starting point is a discrete nonlinear Schr\"odinger equation (DNLS) with general lattice interactions as follows. We consider a 1-d lattice $h \Z$ with mesh size $h > 0$, which  is assumed to be  less than some fixed small constant: $h < h_0 \leq 1$. Moreover, we denote $x_m = hm$ with $m \in \Z$ in the following, and we consider discrete wave functions $u_h : \R \times h \Z \to \C$ that satisfy the  discrete NLS-type equation of the form
\begin{equation}\label{discrete}
i \frac{d}{dt} u_h(t, x_m) = h \sum_{n \neq m} \frac{ u_h(t,x_m) - u_h(t,x_n) }{|x_m-x_n|^{1+2 s} }  \pm |u_h(t,x_m)|^2 u_h(t,x_m) .
\end{equation}
Here $0 < s < \infty$ is a fixed parameter controlling the decay behavior of the lattice interactions. In fact, we will formulate below a generalized version of  problem \eqref{discrete}, where we allow for more general interaction terms of the form $\beta(h)^{-1} J(|n-m|)$, where $J$ is defined below,  in place of the kernel $h(|x_m - x_n|)^{-(1+2s)}$.

Indeed, the discrete NLS equation \eqref{discrete} can be viewed  as  a family of models for quantum particles on a lattice with a three wave interaction set up which gives rise to the cubic nonlinearity, where the $+$ sign represents a repulsive on-site self-interaction and $-$ describes the focusing case. We consider the cubic interaction for simplicity, but what follows  can be easily generalized to different nonlinearities. In terms of DNA, the cubic nonlinearity models a self-interaction for a base pair of the strand with itself, and the summation term models interactions between base pairs decaying like an inverse power of the distance along the strand \cite{MCGJR}. The complex coiling of a DNA strand in three dimensions is what makes it plausible for base pairs to interact with others even a long distance away.

We are interested in the continuum limit, $h \to 0^+$, where we expect that $u_h$ tends (in a weak sense specified below) to a solution $u=u(t,x)$ of the fractional NLS of the form
\begin{equation}\label{continuum}
i \partial_t u = c(-\Delta)^{\alpha} u \pm |u|^2 u
\end{equation}
with $u : \R \times \R \to \C$, a constant $c$ depending only on $s$, and $\alpha$ depending on $s$ appropriately. Here, as usual, the fractional Laplacian $(-\Delta)^\alpha$ on $\R$ is defined via its multiplier $|k|^{2\alpha}$ in Fourier space. Our main results in Theorem \ref{thm:main} below show that the solution $u_h(t,x_m)$ of the discrete equation tends in the limit $h \to 0^+$ to $u = u(t,x)$ solving \eqref{continuum}, where the following holds.

\begin{itemize}
\item For $s$ below $1$ in \eqref{discrete}, the long-range interactions in the discrete NLS-type equation remain long-range in the continuum limit, producing a fractional NLS with a nonlocal character coming from the Laplacian of order $\alpha = s$. 
\item For $s$ above $1$ in \eqref{discrete}, the interaction strength decays quickly enough that only local effects survive in the continuum limit, which is exactly the ``classical"  NLS, $\alpha = 1$. 
\item
For $s=1$ in \eqref{discrete}, we get the classical NLS in the continuum limit, with a logarithmic factor appearing in the scaling constants, see e.g. \eqref{beta} below. 
%\item In fact, we discuss a broader class of discrete equations formulated in \eqref{eq:discivp} below with a general class of lattice interactions.
\end{itemize} 
This should be compared with numerical evidence in the physics literature that says there is a critical value $s_c$, numerically calculated to be near $1$, above which the behavior of the discrete long-range interactions is qualitatively like the (non-fractional) NLS \cite{GMCR1}.

\section{Formulation of the Main Result}

We start by introducing a broad class of discrete evolution equations, thereby generalizing problem \eqref{discrete}.  On the discrete one-dimensional lattice $h \Z$, we consider the evolution problem for the discrete wave function $u_h : [0,T) \times h \Z \to \C$ satisfying the initial value problem
\begin{equation} \label{eq:discivp}
\left \{ \begin{array}{l} \displaystyle i \frac{d}{dt} u_h(t, x_m) =  \frac{1}{\beta(h)} \sum_{n \neq m}  J_{|n-m|} \big [ u_h(t, x_m) - u_h(t, x_n) \big ] \pm |u_h(t,x_m)|^2 u_h(t,x_m)  \\
u_h(0,x_m) = v_h(x_m), \quad \mbox{$x_m = mh$ with $m \in \Z$}  . \end{array} \right . 
\end{equation}
Here we use the notation $J_n=J(|n|)$ to indicate the sequence, $v_h : h \Z \to \C$ is a given initial datum, and $\beta(h) > 0$ denotes the scaling factor depending on the lattice spacing constant $h> 0$. In fact, for a suitable choice of $\beta(h) > 0$ depending on the behavior of $J=(J_n)_{n=1}^\infty$ for large $n$, we will see below that the evolution problem \eqref{eq:discivp} exhibits a reasonable behavior in the {\em ``continuum limit''} as $h \to 0^+$. It turns out that it is natural to assume that $J = (J_n)_{n=1}^\infty$ belongs to the class $\mathcal{K}_s$ for some $0 < s \leq +\infty$, which we define as follows.

\begin{defin}[Interaction of class $\mathcal{K}_s$] Let $J = (J_n)_{n=1}^\infty$ be a sequence with $J_n \geq 0$ for all $n \geq 1$. We say that {\bf $J$ is an $s$-kernel}, or $J \in \mathcal{K}_s$ with $0 < s < \infty$ if 
$$
\mbox{$\displaystyle \lim_{n \to \infty} n^{1+2s} J_n = A$ for some finite $A > 0$.}
$$
Moreover, we say that {\bf $J$ is an $\infty$-kernel}, or $J \in \mathcal{K}_\infty$ if
$$
\mbox{$\displaystyle \lim_{n \to \infty} n^{1+2s} J_n = 0$ for all $s > 0$.}
$$
\end{defin}

\begin{remark} 
1.) The pure-power case $J_n = |n|^{-1-2s}$ clearly satisfies $(J_n)_{n=1}^\infty \in \mathcal{K}_s$. \\
2.) Any $(J_n)_{n=1}^\infty \in \mathcal{K}_\infty$ provided that $J_n \neq 0$ for only finitely many $n$. In particular, the case of {\em nearest-neighbor interactions} when $J_n = 1$ for $n = \pm 1$ and $J_n = 0$ else belongs to $\mathcal{K}_\infty$.  Note also that the class of exponentially decaying $J_n \sim e^{-c n} $ with some $c > 0$ belongs to $\mathcal{K}_\infty$.
\end{remark}

\medskip
Assuming that $J=(J_n)_{n=1}^\infty$ belongs to $\mathcal{K}_s$ for some $0 < s \leq +\infty$, it follows from standard arguments, see Proposition \ref{prop:gwp_discrete} below, that we have global well-posedness  for the initial-value problem \eqref{eq:discivp} in the space $L^2_h$ defined as
$$
L_h^2 = \{ v_h \in \C^{h \Z} : (v_h, v_h)_L^2 := h \sum_{m \in \Z} |v_h(x_m)|^2 < +\infty \}.
$$
In addition, it is straightforward to check that \eqref{eq:discivp} exhibits conservation of energy
$$
E(u_h) = \frac{1}{2 \beta(h)} \sum_{m,n : n \neq m} J_{|n-m|} \big | u_h(x_m) - u_h(x_n) \big |^2 \pm \frac{h}{4} \sum_{m} |u_h(x_m)|^4 ,
$$ 
and conservation of the (discrete) $L^2$-mass given by
$$
N(u_h) =h \sum_{m} |u_n(x_m)|^2 .
$$
Here, the overall factor of $h > 0$ appearing in $E(u_h)$ and $N(u_h)$ is a convenient convention when we discuss the continuum limit when $h \to 0^+$.

Associated to \eqref{eq:discivp}, we now turn to its tentative continuum problem for the wave function $u: [0,T) \times \R \to \C$. More specially, we consider NLS-type initial-value  problems of the form
\begin{equation} \label{eq:fNLS}
\left \{ \begin{array}{l} i \partial_t u = c (-\Delta)^{\alpha} u \pm |u|^2 u , \\
u(0,x) = v(x), \quad u : [0,T) \times \R \to \C. \end{array} \right .
\end{equation}
Here $c >0$ is some fixed constant determined below, and $(-\Delta)^{\alpha}$ denotes the (fractional) Laplacian on $\R$ given by its Fourier multiplier $|k|^{2 \alpha}$, where we assume that $0 < \alpha \leq 1$ holds in what follows. Note that $\alpha=1$ corresponds to the ``classical'' NLS, whereas the range $0 < \alpha < 1$ can be regarded as ``fractional'' NLS. In the focusing case when the minus sign stands in front of the nonlinearity in \eqref{eq:fNLS}, there exist ground state solitary wave solutions. For uniqueness (and further properties) of such ground states, we refer to \cite{FrLe}.  

Regarding the well-posedness for \eqref{eq:fNLS}, we record the following simple fact.

\begin{prop}\label{NLScont}
Suppose that $1/2 < \alpha \leq 1$, $c>0$, and let $v \in H^\alpha(\R)$ be an initial datum for \eqref{eq:fNLS}. Then  the initial-value problem \eqref{eq:fNLS} has  a global unique solution $u \in C^0([0,\infty); H^\alpha(\R))$. Moreover, we have conservation of energy and $L^2$-mass given by
\begin{equation}\label{energy}
\mathcal{E}(u) = \frac{c}{2} \int_\R \overline{u} (-\Delta)^{\alpha} u \pm \frac{1}{4} \int_\R  |u|^4, \quad \mathcal{N}(u) = \int_\R |u|^2. 
\end{equation}
Finally, we have the following global a-priori bound
\begin{equation}\label{mass}
\| u \|_{L^\infty_t H^{\alpha}_x} \leq C(u(0)).
\end{equation}
\end{prop}

\begin{proof}
Thanks to the Sobolev embedding $H^\alpha(\R) \subset L^\infty(\R)$ when $\alpha > \frac 1 2$, this follows standard arguments of abstract evolution equations; see, e.\,g., \cite{C}. Indeed, by a simple fixed point argument, we deduce existence and uniqueness of $u \in C^0([0,T]; H^\alpha(\R))$ solving \eqref{eq:fNLS} for $T > 0$ sufficiently small, by using the integral equation
$$
u(t) = e^{-it(-\Delta)^{\alpha} } v - i \int_0^t e^{-i(t-s) (-\Delta)^{\alpha}} |u(s)|^2 u(s) \, ds.
$$
Note that the map $u \mapsto |u|^2 u$ is locally Lipschitz on $H^\alpha(\R) \subset L^\infty(\R)$ when $\alpha > \frac 1 2$. This shows local well-posedness for \eqref{eq:fNLS}. Moreover, it easy to check that $\mathcal{E}(u)$ and $\mathcal{N}(u)$ are conserved quantities. Finally, the global a-priori bound $\sup_{t \geq 0} \| u \|_{H^{\alpha}} \leq C(u(0))$ follows from conservation of $\mathcal{E}(u)$ and $\mathcal{N}(u)$ combined with the fractional Gagliardo--Nirenberg inequality
$$
\int_\R |  u |^4 \leq C \left ( \int_\R | (-\Delta)^\frac{\alpha}{2} u |^2 \right )^{\frac 1 {2 \alpha}} \left ( \int_\R | u |^2 \right )^{2-\frac 1 {2 \alpha}} .
$$
Using now the a-priori bound $\sup_{t \geq 0}\| u \|_{H^\alpha} \leq C$, we deduce that any local solution $u \in C^0([0,T]; H^\alpha(\R))$ extends to all times $t \geq 0$. \end{proof}

\begin{remark}
The above well-posedness result for \eqref{eq:fNLS} can be easily generalized to power-type nonlinearities $f(u) = \pm |u|^{2 \sigma} u$ with $0 < \sigma < +\infty$, instead of $\pm |u|^2 u$. More precisely, one obtains local well-posedness in $H^\ell(\R)$ with $\ell \geq \alpha$, and the solution $u \in C^0([0,T); H^\ell(\R))$ extends globally in time in the case when $f(u) = - | u|^{2\sigma} u$ with $0 < \sigma < 2 \alpha$ (focusing $L^2$-subcritical case) or when $f(u) = +|u|^{2 \sigma} u$ with any $0 < \sigma < +\infty$ (defocusing case). 
\end{remark}

Let us now formulate the main result. Given a lattice function $f_h : h \Z \to \C$, we define (see also \cite[Chapter V]{L} and \cite{SSB}) its piecewise linear interpolation $p_h : \R \to \C$ to be given by
\begin{equation} \label{def:ph}
(p_h f_h)(x) := f_h(x_m) + (D^+_h f_h)(x_m) (x-x_m), \quad \mbox{for $x \in [x_m, x_{m+1})$}.
\end{equation}
Here $D^+_h$ denotes the discrete right-hand derivative on $h \Z$ defined as
$$
(D^+_h f_h)(x_m) := \frac{f(x_{m+1}) - f(x_m)}{h} .
$$
On the other hand, given a locally integrable function $f : \R \to \C$, we define its discretization $f_h : h \Z \to \C$ by setting
\begin{equation} \label{eq:discretize}
f_h(x_m) := \frac{1}{h} \int_{x_m}^{x_{m+1}} f(x) \, dx, \quad \mbox{with $x_m = mh$ and $m \in \Z$}.
\end{equation}
It is easy to see that $\| f_h \|_{L^2_h} \leq \| f \|_2$ (see Lemma \ref{prop:unifH2}). Moreover, as we will detail below (see Lemma 
\ref{prop:unifH2} and Lemma \ref{lem:interbound}) some straightforward calculations combined with interpolation theory show that 
\begin{equation*}
\| p_h f_h \|_{H^\alpha} \leq C \| f \|_{H^\alpha} ,
\end{equation*}
for every $0 \leq \alpha \leq 1$, where $C > 0$ is some constant independent of $h >0$.

The main result of this paper now reads as follows. 

\begin{thm}[Continuum Limit] \label{thm:main}
Let $J = ( J_n )_{n=1}^\infty \in \mathcal{K}_s$ for some $\frac 1 2 < s \leq +\infty$, where we assume that $J_1 > 0$ holds.\footnote{This is a convenient and physically reasonable assumption, saying that at least neighboring lattice sites interact.} Furthermore, we define 
$$
\alpha := \left \{ \begin{array}{ll} s, & \quad \mbox{for $\frac 1 2 < s < 1$}, \\
1 , & \quad \mbox{for $s \geq 1$} . \end{array} \right .
$$
Now suppose that $v \in H^\alpha(\R)$ and consider its discretization $v_h : h \Z \to \C$ defined as in \eqref{eq:discretize}. Finally, let $u_h = u_h(t, x_m)$ denote the corresponding unique global solution to \eqref{eq:discivp} with initial datum $v_h \in L^2_h$ given by Proposition \ref{prop:gwp_discrete} below, where we choose
\begin{equation}\label{beta}
\beta(h) := \left \{ \begin{array}{ll} h^{2s}, & \quad \mbox{for $\frac 1 2 < s < 1$}, \\
(- \log h) h^2, & \quad \mbox{for $s=1$},\\
h^{2}, & \quad \mbox{for $s > 1$}. \end{array} \right . 
\end{equation}
Then, for every $0 < T < +\infty$ fixed, we have the convergence 
$$
\mbox{$p_h u_h \weakto u$ weakly-$*$ in $L^\infty ( [0,T]; H^\alpha(\R) )$ as $h \to 0^+$} . 
$$
Here $u \in C^0([0,\infty); H^\alpha(\R))$ is the unique global solution of the initial-value problem \eqref{eq:fNLS} with $\alpha >\frac{1}{2}$ defined above and some constant $c > 0$ that only depends on $J$.  \end{thm}

%\begin{remark}\label{overage}
%With regard to the discretization \eqref{eq:discretize}, it is easy to see that $v_h \to v$ strongly in $L^2(\R)$ as $h \to 0^+$. 

%\end{remark}

%\begin{remark}Probably a stringer version of Theorem \ref{thm:main} would be to prove the convergence while staying in $H^\alpha$ instead of invoking the $H^1$ norm. This would be indeed more natural since the  the energy in \eqref{energy}  is of order $\alpha$. On the other hand we couldn't find an appropriate interpolation as in 
%\eqref{def:ph} while staying in the $H^\alpha$ space, hence our assumption.
%\end{remark}

\begin{remark}\label{dimension}
We also expect a similar weak-$*$ convergence result in the range $\frac 1 4 \leq s  \leq \frac 1 2$. Note that for $s <\frac 1 4$, the cubic nonlinearity $|u|^2 u$ becomes supercritical. Moreover, for $s=\frac 1 2$ the problem \eqref{eq:fNLS} becomes $L^2$-critical and thus a smallness condition on the initial datum $v$ must be imposed to have global well-posedness in $H^{\frac 1 2}(\R)$; see \cite{Ge}. Also note that for $0 < s < \frac 1 2$, it is presently not known whether the initial value problem \eqref{eq:fNLS} is locally well-posed in $H^s(\R)$. In particular, the continuum limit may depend on the chosen subsequence $h_n \to 0$. Furthermore, as a related problem, it would be desirable to understand the case of higher space dimensions $d \geq 2$. 

In some sense, some arguments  we use below hinge on the fact that we require $H^s(\R^d) \subset L^\infty(\R^d)$, which forces us to assume that $s > \frac{1}{2}$ and $d=1$ at the moment. We leave the extension to higher space dimensions or small $s$ as an interesting open problem.
\end{remark}

%\begin{remark}
%Note that it follows from standard arguments that if $v \in H^\alpha(\R)$ then $p_h v_h \weakto v$ weakly in $H^\alpha(\R)$ as $h \to 0^+$. In particular, we have weak convergence of the initial data  $p_h u_h(0) \weakto v$ weakly in $H^\alpha(\R)$ as $h \to 0^+$.
%\end{remark}

\subsection*{Plan of the Paper}
This paper is organized as follows. In Section \ref{sec:prelim}, we introduce a class of fractional Sobolev type norms on the discrete lattice $h \Z$. Moreover, we prove some uniform embedding and interpolation estimates that are uniform with respect to the lattice constant $h \in (0,  h_0]$ with $0 < h_0 < 1$ being some fixed constant. In Section \ref{discretevolution}, we discuss the discrete evolution problem \eqref{eq:discivp}. Furthermore, we derive a-priori bounds for  $u_h=u_h(t,x_m)$ that are uniform in $h \in (0,h_0]$. Finally,  we prove Theorem \ref{thm:main} in Section \ref{sec:proof} by convergence results for the discrete equation derived below, combined with a suitable weak compactness arguments (inspired by the work of Sulem-Sulem-Bardos \cite{SSB} on the Landau-Lifshitz equation). 

In Appendix A--C, we collect and prove some technical results needed in this paper. 

\subsection*{Acknowledgments}
We are grateful to the anonymous referee for carefully reading our manuscript and improving our paper.

\section{Preliminaries}  \label{sec:prelim}

In this section, we state and prove some technical results that will be needed in the proof of Theorem 1. Throughout this section, we suppose that $0 < h_0 < 1$ is a fixed constant and we consider the family of lattices $h \Z$ with $h \in (0, h_0]$. All constants $C > 0$ appearing in the following inequalities can be chosen to depend only on $h_0 > 0$. 

\subsection{Discrete uniform Sobolev inequalities} 

In the following, we denote $x_m=mh$ with $m \in \Z$. For sequences $u_h, v_h \in \C^{h \Z}$, we define the inner product and norm 
$$
(v_h, u_h)_{L^2_h} := h \sum_{m \in \Z} \overline{u_h(x_m)} v_h(x_m) , \quad \| u_h \|_{L^2_h}^2 := (u_h,u_h)_{L^2_h},
$$ 
and we set $L^2_h := \{ u_h \in \C^{h \Z}: \| u_h \|_{L^2_h} < +\infty \}$. For $u_h \in L^2_h$, we define its Fourier transform $\hat{u}_h : [-\pi, \pi] \to \C$ by
$$
\hat{u}_h( k ) := \frac{1}{\sqrt{2 \pi}} \sum_{m \in \Z} u_h(x_m) e^{-i m k} .
$$
Since $u_h \in L^2_h$, we have $\hat{u}_h \in L^2([-\pi, +\pi])$. Moreover, we have the inversion formula
$$
u_h(x_m) = \frac{1}{\sqrt{2 \pi}} \int_{-\pi}^{+\pi} \hat{u}_h(k) e^{i m k} \, d k,
$$
and Parseval's identity gives us 
$$
(v_h, u_h)_{L^2_h} = h \int_{-\pi}^{+\pi} \overline{\hat{v}_h(k)} \hat{u}_h(k) \, d k .
$$
Using this observation, we introduce the following fractional Sobolev type norm for lattice functions $u_h \in L^2_h$. Let $0 \leq \sigma \leq 1$ be given. We define the norm $\| u_h \|_{H^\sigma_h}$ for $u_h \in L^2_h$ by setting
\begin{equation} \label{def:Hh_sigma}
\| u_h \|_{H^\sigma_h}^2 := h \int_{-\pi}^{+\pi} \big (1 + h^{-2\sigma} |k|^{2\sigma}  \big ) |\hat{u}_h(k)|^2 \, d k
\end{equation}
Clearly, we have $\| u_h \|_{H^0_h} = \| u_h \|_{L^2_h}$. Also, we note that $\| u_h \|_{H^\sigma_h} < +\infty$ for any $u_h \in L^2_h$. However, we shall need precise uniform bounds as $h \to 0^+$. 

\begin{remark} \label{rem:inter}
Note that $\| u_h \|_{H^\sigma_h} \leq C \|u_h \|_{H^\rho_h}$ for $0 \leq \sigma \leq \rho \leq 1$, where the constant $C > 0$ is independent of $h >0$. Furthermore, by a simple interpolation argument, we deduce that 
$$
\| u_h \|_{H^{\sigma_0}_h } \leq \| u_h \|_{H^\sigma_h}^{\sigma_0/\sigma} \| u_h \|_{L^2_h}^{1- \sigma_0/\sigma}
$$
for $0 \leq \sigma_0 \leq \sigma \leq 1$.
\end{remark}

We have the following (discrete) Sobolev estimate that is uniform in $h > 0$.

\begin{lemma}[Discrete uniform Sobolev inequality]  \label{lem:sobo}
For every $\frac 1 2 < \sigma \leq 1$, there exists a constant $C=C(\sigma) >0$ independent of $h > 0$ such that 
$$
\| u_h \|_{L^\infty_h} \leq C \| u_h \|_{H^\sigma_h}
$$
for all $u_h \in L^2_h$. Here $\| u_h \|_{L^\infty_h} = \sup_{m \in \Z} |u(x_m)|$. 
\end{lemma}

\begin{proof}
By the Fourier inversion formula and the Cauchy--Schwarz inequality,
\begin{align*}
\| u_h \|_{L^\infty_h} &  \leq \frac{1}{\sqrt{2 \pi}} \int_{-\pi}^{+\pi} |\hat{u}_h(k)| \, d k \leq C \left ( \int_{-\pi}^{+\pi} \frac{d k}{1+ h^{-2\sigma} |k|^{2\sigma}} \right )^{1/2} h^{-1/2} \| u_h \|_{{H}^\sigma_h}  \\
& \leq C \left ( h \int_{-\infty }^{+\infty} \frac{d z}{1+ |z|^{2\sigma}} \right )^{1/2} h^{-1/2} \| u_h \|_{H^{\sigma}_h} \leq C \| u_h \|_{H^{\sigma}_h} ,
\end{align*}
with some finite constant $C= C(\sigma) > 0$ independent of $h> 0$.\end{proof}

Next, we prove the following discrete Gagliardo--Nirenberg type inequality uniform with respect to $h >0$. 

\begin{lemma}[Discrete uniform Gagliardo--Nirenberg inequality] \label{lem:GN}
Define the discrete norm $\| u_h \|_{L^4_h} = ( h \sum_{m} |u_h(x_m)|^4)^{\frac 1 4 }$. Then for any $\frac 1 4 < \sigma_0 \leq 1$, there exists a constant $C=C(\sigma_0) > 0$ independent of $0 < h \leq h_0$ such that
$$
\| u_h \|_{L^4_h} \leq C \| u_h \|_{H^\sigma_h}^{\sigma_0/\sigma} \| u_h \|_{L^2_h}^{1-\sigma_0/\sigma},$$
for every $\sigma_0 \leq \sigma \leq 1$.
\end{lemma}

\begin{proof}
Using the Hausdorff--Young inequality and H\"older's inequality, we conclude that
\begin{align*}
\left ( \sum_{m} |u_h(x_m)|^4  \right )^{1/4} & \leq C \left ( \int_{-\pi}^{+\pi} |\hat{u}_h(k)|^{4/3} \, d k \right )^{3/4} \\
& \leq C \left ( \int_{-\pi}^{+\pi} (1+ h^{-2\sigma_0} |k|^{2\sigma_0} ) |\hat{u}_h(k)|^2 \, d k \right )^{1/2} \\
& \quad \cdot \left (  \int_{-\pi}^{+\pi} \frac{d k }{(1+ h^{-2\sigma_0} |k|^{2\sigma_0} )^{2}}   \right )^{1/4} \\
& \leq C h^{-1/2} \| u_h \|_{H^{\sigma_0}_h} \left ( h  \int_{-\infty}^{+\infty} \frac{d z}{(1+|z|^{2\sigma_0})^2} \right )^{1/4} \\
&  \leq C h^{-1/4} \| u_h\|_{H^{\sigma_0}_h},
\end{align*}
where $C = C(\sigma_0) > 0$ is independent of $h> 0$. This is the desired estimate when $\sigma=\sigma_0$ holds. To complete the proof of Lemma \ref{lem:GN} for $\sigma_0 < \sigma \leq 1$, we simply use the interpolation estimate from Remark \ref{rem:inter} above. \end{proof}

\subsection{Discrete Energy Norm and Estimates} 
Throughout this subsection, we assume that the sequence nonnegative numbers $J= ( J_n )_{n=1}^\infty$ satisfies the following conditions:
\begin{itemize}
\item[(A1)] $J \in \mathcal{K}_s$ for some $0 < s \leq +\infty$;
\item[(A2)] $J_1 > 0$.
\end{itemize} 
Next, we define the linear operator $\Ldif^{J}_h : L^2_h \to L^2_h$ by setting
\begin{equation} \label{def:Jop}
(\Ldif_h^{J} u_h)(x_m) :=  \frac{1}{\beta(h)} \sum_{n \neq m}  J_{|m-n|} \big [ u_h(x_m) - u_h(x_n)  \big ] 
\end{equation}
Here and in what follows, we choose $\beta(h)$ to be given by
\begin{equation} \label{eq:beta}
\beta(h) = \left \{ \begin{array}{ll} h^{2s}, & \quad \mbox{if $J \in \mathcal{K}_s$ with some $0 < s < 1$}, \\
(- \log h) h^2, & \quad \mbox{if $J \in \mathcal{K}_s$ with $s=1$},\\
h^{2}, & \quad \mbox{if $J \in \mathcal{K}_s$ with some $s > 1$}. \end{array} \right . 
\end{equation}
Note that we always impose (without loss of generality) that $0 < h \leq h_0 < 1$ holds.  In particular, we have that $\beta(h) > 0$ is positive. By changing the summation index, we deduce that
\begin{equation} \label{eq:Jgood}
(\Ldif_h^{J} u_h)(x_m) =  \frac{1}{\beta(h)} \sum_{n \neq 0}  J_{|n|} \big [ u_h(x_m) - u_h(x_m - x_n) \big ].
\end{equation}
Clearly, the operator $\Ldif^{J}_h$ is bounded on $L^2_h$ with $\| \Ldif^{J}_h u_h \|_{L^2_h} \leq C \beta(h)^{-1} \| u_h \|_{L^2_h}$, using that the sum $\sum_{n =1}^\infty J_n$ is finite. Also, we easily check that $\Ldif^{J}_h$ is self-adjoint; that is, we have that $(\Ldif^{J}_h)^* = \Ldif^{J}_h$ holds on $L^2_h$. Furthermore, a simple calculation shows 
\begin{equation} 
(u_h, \Ldif^{J}_h u_h )_{L^2_h} = \frac{1}{2 \beta(h)} \sum_{n,m : n \neq m}   J_{|n|} \big | u_h(x_m) - u_h(x_m - x_n) \big |^2.
\end{equation}
Since $J_n \geq 0$ by assumption, this shows that $\Ldif^{J}_h \geq 0$ is nonnegative as an operator. In particular, for any $u_h \in L^2_h$, we can define the norm 
\begin{equation}\label{energynorm}
 \| u_h \|_{H^{J}_h}^2 := (u_h, u_h)_{L^2_h} + (u_h, \Ldif^{J}_h u_h )_{L^2_h} 
\end{equation}
Below, the norm $\| \cdot \|_{H^{J}_h}$ will play the role of an energy norm for the discrete evolution problem \eqref{eq:discivp}. We have the following norm equivalence uniform in $h > 0$. 

\begin{lemma}[Uniform Norm Equivalence] \label{lem:embed}
Suppose $J = ( J_n )_{n=1}^\infty$ satisfies (A1) and (A2) above for some $0 < s \leq +\infty$ with $s \neq 1$. Let
 \begin{equation}\label{alpha}
\alpha = \left \{ \begin{array}{ll} s, & \quad \mbox{for $0 < s < 1$}, \\
1 , & \quad \mbox{for $s \geq 1$} . \end{array} \right .
\end{equation}
%$$
%\beta :=  \left \{ \begin{array}{ll} 2s & \quad \mbox{for $0 < s < 1$} \\
%2 & \quad \mbox{for $s \geq 1$} \end{array} \right .  ,\quad  \sigma :=  \left \{ \begin{array}{ll} s & \quad \mbox{for $0 < s < 1$} \\
%1 & \quad \mbox{for $s \geq 1$} \end{array} \right .  .
%$$
Then there exist constants $A, B > 0$ independent of $h >0$ such that 
$$
A \| u_h \|_{H^{\alpha}_h} \leq \| u_h \|_{H^{J}_h} \leq B \| u_h \|_{H^{\alpha}_h},
$$
for all $u_h \in  L^2_h$.  
\end{lemma}

\begin{proof}
By using \eqref{eq:Jgood} together with the Fourier inversion formula and Parseval's formula, we notice that
$$
(u_h, \Ldif^{J}_h u_h)_{L^2_h} = h \int_{-\pi}^{+\pi} \frac{\omega(k)}{\beta(h)} |\hat{u}_h(k)|^2 \, dk, 
$$
where
$$
\omega(k) = 2 \sum_{n=1}^\infty J_n \big [ 1- \cos(nk) \big ] .
$$
In view of the definition of $\| \cdot \|_{H^{\alpha}_h}$ it remains to show that, for $|k| \leq \pi$ and $h > 0$, 
\begin{equation} \label{ineq:embed}
A \left (1 + h^{-2\alpha} |k|^{2\alpha} \right ) \leq \left (1 + \frac{ \omega(k) }{\beta(h)} \right ) \leq B \left (1 + h^{-2\alpha} |k|^{2\alpha} \right ),
\end{equation}
for some constants $A, B > 0$ independent of $h > 0$, where we define $\alpha = s$ for $0 < s < 1$ and $\alpha =1$ for $s>1$.

Le us first prove the lower bound in inequality \eqref{ineq:embed}. From Lemma \ref{lem:asymp}, we recall that
\begin{equation} \label{ineq:ks}
\frac{C}{2} |k|^{2 \alpha} \leq \omega(k) \leq  C |k|^{2 \alpha}, \quad \mbox{for $|k|  \leq k_0$},
\end{equation}
with some constant $C > 0$ and $k_0 > 0$ sufficiently small. Furthermore, using that $(1- \cos z) \geq 0$ for all $z \in \R$ and $(1-\cos z) \geq \frac{2}{\pi^2} z^2$ for $|z| \leq \pi$, we find that
\begin{equation} \label{ineq:ks2}
\omega(k) = 2 \sum_{n=1}^\infty J_n \big [ 1- \cos (nk) \big] \geq \frac{4 J_1}{\pi^2} k^2 , \quad \mbox{for $|k| \leq \pi$}.
\end{equation}
Note that $J_1 > 0$ by assumption. Combining the lower bound in \eqref{ineq:ks} with \eqref{ineq:ks2} and using that $\alpha \leq 1$, we infer that
\begin{equation} \label{ineq:ks3}
\omega(k) \geq \delta |k|^{2\alpha}, \quad \mbox{for $|k| \leq \pi$}.
\end{equation}
Here $\delta  = \min \{ \frac{C}{2}, \frac{4 J_1}{\pi^2} |k_0|^{2-2\alpha} \} > 0$ with $C >0$ taken from \eqref{ineq:ks}. Using \eqref{ineq:ks3} and recalling that $\beta(h) = h^{2 \alpha}$, we derive
$$
1 + \frac{\omega(k)}{\beta(h)} \geq 1+ \delta h^{-2 \alpha}|k|^{2 \alpha} \geq A \left ( 1+ h^{-2\alpha} |k|^{2 \alpha} \right ),
$$
for all $|k| \leq \pi$ and $h >0$, where we also used that $(1+\delta |z|^{2\alpha}) \geq A (1+ |z|^{2\alpha})$ for all $z \in \R$ with the constant $A = \min \{ \delta, 1 \} >0$. This shows that the lower bound in \eqref{ineq:embed} holds.

To prove the upper bound in \eqref{ineq:embed} in the case $s \neq 1$, we argue as follows. First, by the upper bound in \eqref{ineq:ks}, we conclude that, for all $h >0$, 
\begin{equation} \label{ineq:ks4}
\frac{\omega(k)}{\beta(h)} \leq C h^{-2 \alpha} |k|^{2\alpha} \leq C (1+ h^{-2 \alpha} |k|^{2\alpha}), \quad \mbox{for $|k| \leq k_0$}.
\end{equation} 
On the other hand, we recall the global upper bound $\omega(k) \leq 4 \sum_{n=1}^\infty J_n \leq C$ for $|k| \leq \pi$. Hence we find that, for any $h > 0$, 
\begin{equation} \label{ineq:ks5}
\frac{\omega(k)}{\beta(h)} \leq C h^{-2 \alpha} \leq C ( 1 + h^{-2 \alpha} |k_0|^{2 \alpha}) \leq C (1 + h^{-2 \alpha} |k|^{2 \alpha}) , \quad \mbox{for $|k| \geq k_0$} .
\end{equation}
Combining now \eqref{ineq:ks4} and \eqref{ineq:ks5}, we deduce that the upper bound in \eqref{ineq:embed} holds. This completes the proof of Lemma \ref{lem:embed}. \end{proof}

Next, we treat the special case when $J = (J_n)_{n=1}^\infty$ belongs to $\mathcal{K}_s$ with $s=1$. 

\begin{lemma}[Norm equivalence for $s=1$] \label{lem:embedlog}
Suppose that $J=(J_n)_{n=1}^\infty$ satisfies (A1) and (A2) above with $s=1$. Then we have 
$$
\| u \|_{H^J_h} \leq C \| u \|_{H^1_h},
$$
with some constant $C > 0$ independent of $h >0$. Moreover, for every $0 < \sigma < 1$ and $h_0 > 0$ sufficiently small, there exists a constant $K > 0$ independent of $0 <h \leq h_0$ such that
$$
\| u \|_{H^\sigma_h} \leq K \| u \|_{H^J_h} .
$$
\end{lemma}

\begin{proof}
Similar as in the proof of Lemma \ref{lem:embed}, we have to show that
\begin{equation} \label{ineq:bordercase}
A \left ( 1 + h^{-2 \sigma} |k|^{2 \sigma} \right ) \leq \left (1 + \frac{\omega(k)}{\beta(h)} \right ) \leq B \left ( 1 + h^{-2} |k|^2 \right ), \quad \mbox{for $|k| \leq \pi$},
\end{equation}
with some constant $A > 0$, which may depend on $0 < \sigma < 1$, and some constant $B > 0$ independent of $h > 0$. Recall that $\beta(h) = -(\log h) h^2 > 0$, since we can assume that $0 < h \leq h_0 < 1$ holds. For later use, we recall from Lemma \ref{lem:asymp} the bound
\begin{equation} \label{ineq:bordercase2}
- c (\log |k|) |k|^2 \leq \omega(k) \leq - C (\log |k|) |k|^2, \quad \mbox{for $|k| \leq k_0$},
\end{equation}
where $c > 0$ and $C > 0$ are some constants and $k_0 > 0$ is sufficiently small. Note that $c > 0$, $C > 0$ and $k_0 > 0$  only depend on $J=(J_n)_{n=1}^\infty$. 

%Note that we can assume that $0 < h_0 < k_0 < 1$ holds. Hence, we have that
%\begin{equation} \label{ineq:bordercase2}
%- \frac{C}{2} (\log |k|) |k|^2 \leq \omega(k) \leq - C (\log |k|) |k|^2, \quad \mbox{for $|k| \leq h$},
%\end{equation}
%for all $0 < h \leq h_0$. 

First, we prove the upper bound in \eqref{ineq:bordercase}. If we now let $z = h^{-1} |k|$ with $|k| \leq k_0$, then the upper bound in \eqref{ineq:bordercase2} gives us 
\begin{align}
\frac{\omega(k)}{\beta(h)} & \leq C \frac{ \left | \log ( h z) \right | }{\left | \log h \right | } z^2 = C \frac{\left | \log z + \log h \right |}{\left | \log h \right | } z^2 \leq  C \left ( \frac{ \left | \log z \right |}{\left |\log h \right |} z^2 + z^2  \right ) 
\end{align}
we will show that 
\begin{equation} \label{ineq:logtrick}
\frac{\left | \log z \right | }{\left |\log h \right |} z^2  \leq C (1 + z^2), \quad \mbox{for $z \leq h^{-1} k_0$ and $0 < h \leq h_0 < 1$},
\end{equation}
where $C > 0$ only depends on $k_0 >0 $ and $h_0 > 0$, 
which allows us to conclude that 
\begin{equation} \label{ineq:uplog1}
\frac{\omega(k)}{\beta(h)} 
 \leq C \left (1 + z^2 \right ) = C \left (1 + h^{-2} |k|^2 \right ), \quad \mbox{for $|k| \leq k_0$},
\end{equation}
with some constant $C >0$ that only depends on $h_0 > 0$ and $k_0 > 0$.
To show \eqref{ineq:logtrick}, we first note that $\left | \log z \right | z^2 \leq C$ for $z \leq 1$ and hence
$$
\frac{\left | \log z \right | }{\left |\log h \right |} z^2 \leq \frac{C}{\left | \log h \right | } \leq \frac{C}{\left |\log h_0 \right |},   \quad \mbox{for $z \leq 1$} ,
$$
using also that $\left | \log h \right |^{-1} \leq \left | \log h_0 \right |^{-1}$ for $0 < h \leq h_0 <1$. On the other hand, we have that $z \mapsto |\log z|$ is monotone increasing on the interval $[1, h^{-1} k_0]$. Therefore, 
\begin{align*}
\frac{\left | \log z \right | }{\left |\log h \right |} z^2 & \leq \frac{ \left | \log (h^{-1} k_0) \right | }{\left | \log h \right |} z^2 =   \frac{ \left | \log h - \log k_0 \right | }{\left | \log h \right |} z^2 \\
&  \leq \left  (1 + \frac{\left | \log k_0 \right |}{\left | \log h_0 \right |} \right ) z^2 \leq C z^2,  \quad \mbox{for $1 \leq z \leq h^{-1} k_0$} .
\end{align*}
Combining the previous estimates, we see that \eqref{ineq:logtrick} follows.

To complete the proof of the upper bound in \eqref{ineq:bordercase}, we recall that $\omega(k) \leq C$ for $|k| \leq \pi$. This yields that
\begin{align} \label{ineq:uplog2}
\frac{\omega(k)}{\beta(h)} & \leq \frac{C}{ \left | \log h_0 \right | h^2} \leq C h^{-2} \leq C \left (1 + h^{-2} |k_0|^2 \right ) \\
& \leq C \left  (1 + h^{-2} |k|^2 \right ) , \quad \mbox{for $|k| \geq k_0$}, \nonumber
\end{align}
using again that $\left | \log h \right |^{-1} \leq \left | \log h_0 \right |^{-1}$. From \eqref{ineq:uplog1} and \eqref{ineq:uplog2}, we deduce that the upper bound in \eqref{ineq:bordercase} holds. 

It remains to establish the lower bound in \eqref{ineq:bordercase}, which is slightly more tedious. We argue as follows. First, we notice that we can assume $0 < h_0 < 1$ satisfies $(h_0)^{\frac 1 2} \leq k_0$, where $k_0 > 0$ is the constant in \eqref{ineq:bordercase2}. Recalling \eqref{ineq:bordercase2} and using the fact that 
$|k| \mapsto \left | (  \log |k|  ) \right |$ is monotone decreasing on the interval $(0, h^{\frac 1 2}]$ where $h < 1$, we obtain  the lower bound
$$
\frac{\left | ( \log |k| ) \right |}{\left | \log h \right |} \geq \frac{\left | \log (h^{1/2}) \right |}{\left | \log h \right |} = \frac{1}{2}, \quad \mbox{for $|k| \leq h^{\frac 1 2}$}.
$$
As a consequence  for all $0 < h \leq h_0$, 
\begin{equation} \label{ineq:loglower1}
\frac{\omega(k)}{\beta(h)} \geq \frac{c}{2} h^{-2} k^2, \quad \mbox{for $|k| \leq h^{\frac 1 2}$}.
\end{equation}
Next, by the same argument as in the proof of Lemma \ref{lem:embed}, we have the general lower bound $\omega(k) \geq C k^2$ for $|k| \leq \pi$ with some constant $C >0$. Therefore, 
\begin{equation} \label{ineq:loglower2}
\frac{\omega(k)}{\beta(h)} \geq \frac{C}{\left | \log h \right |} h^{-2} k^2, \quad \mbox{for $|k| \leq \pi$}. 
\end{equation}
Note that $-(\log h) = \left | \log h \right |$, since in particular $h < 1$. Next, let $0 < \sigma < 1$ be given. We claim that, for $h_0 > 0$ sufficiently small, there exists a constant $A > 0$ (depending only on $h_0 > 0$ and $0 < \sigma <1$) such that
\begin{equation} \label{ineq:logtrick3}
\frac{1}{\left | \log h \right |} h^{-2} k^2 \geq A h^{-2\sigma} |k|^{2\sigma}, \quad \mbox{for $|k| \geq h^{\frac 1 2}$}.
\end{equation} 
Indeed, let $\eps = 1- \sigma$. For $|k| \geq h^{\frac 1 2}$, we see that
$$
\frac{1}{\left | \log h \right |} \frac{ h^{-2} k^2}{ h^{-2 \sigma} |k|^{2 \sigma}} = \frac{h^{-2 \eps} |k|^{2 \eps}}{\left | \log h \right |}  \geq \frac{ h^{-2 \eps} h^{\eps}}{\left | \log h \right |} = \frac{h^{-\eps}}{\left | \log h \right |} =: f(h) .
$$
Since $\eps > 0$, we see that $f(h) \to +\infty$ as $h \to 0^+$. Hence, by choosing $h_0 > 0$ sufficiently small, we obtain that
$$
f(h) \geq A , \quad \mbox{for $0 < h \leq h_0$},
$$
where $A > 0$ is some positive constant that depends only on $h_0 > 0$ and $0 < \sigma < 1$. This proves that estimate \eqref{ineq:logtrick3} holds. In view of \eqref{ineq:loglower2}, we deduce that
\begin{equation} \label{ineq:loglower3}
\frac{\omega(k)}{\beta(h)} \geq C h^{-2 \sigma} |k|^{2 \sigma}, \quad \mbox{for $h^{\frac 1 2} \leq |k| \leq \pi$},
\end{equation}
where $C >0$ only depends on $h_0 > 0$ and $0 < \sigma < 1$. 

Finally, we recall \eqref{ineq:loglower1} and deduce that
\begin{equation} \label{ineq:loglower4}
1+ \frac{\omega(k)}{\beta(h)} \geq C \left ( 1 + h^{-2\sigma} |k|^{2 \sigma} \right ), \quad \mbox{for $|k| \leq h^{\frac 1 2}$},
\end{equation}
where we use that $1+ \delta t^2 \geq C ( 1+ |t|^{2 \sigma})$ for all $t \in \R$ with $0 < \sigma < 1$ and $\delta > 0$, where $C > 0$ only depends on $\sigma$ and $\delta$. Combining now \eqref{ineq:loglower4} and \eqref{ineq:loglower3}, we conclude that the lower bound in \eqref{ineq:bordercase} holds. This completes the proof of Lemma \ref{lem:embedlog}. \end{proof}

Next, we proceed to study the relation of the scale of discrete Sobolev type norms $\| \cdot \|_{H^s_h}$ to the following (classical) discrete Sobolev norm given by
\begin{equation}\label{classicaldiscrete}
 \| u_h \|_{\tilde{H}^1_h}^2 := (u_h, u_h)_{L^2_h} + (D^+_h u_h, D^+_h u_h)_{L^2_h} 
\end{equation}
Here $D^+_h$ denotes the discrete right-hand derivative on the lattice $h \Z$, i.\,e., we have
$$
(D^+_h u_h)(x_m) := \frac{ u_h(x_{m+1})-u_h(x_m)  }{h} .
$$
For later use, we also derive the following uniform embedding estimate.

\begin{lemma} \label{lem:hfull}
For every $0 < \sigma \leq 1$, there exists a constant $C = C(s) > 0$ independent of $h >0$ such that
$$
\| u_h \|_{H^\sigma_h} \leq C \| u_h \|_{\tilde{H}^1_h} , \quad \| u_h \|_{\tilde{H}^1_h} \leq C \| u_h \|_{H^1_h},
$$
for all $u_h \in L^2_h$.
\end{lemma}

\begin{proof}
Using the Fourier transform and Parseval's identity, we find that
$$
\| u_h \|_{\tilde{H}^1_h}^2 =  h \int_{-\pi}^{+\pi} \big ( 1+ h^{-2} \Omega(k) \big )  |\hat{u}_h(k)|^2 \, d k
$$
where
$$
\Omega(k) := 2 - 2 \cos ( k ) = 4 \sin^2 (k/2).
$$
The claimed bound follows from Parseval's identity, provided we can show that
\begin{equation*} 
(1+ h^{-2 \sigma} |k|^{2\sigma} )  \leq C ( 1+ h^{-2} \Omega(k) ), \quad \mbox{for $|k| \leq \pi$ and $h > 0$}.
\end{equation*}
with some constant $C=C(\sigma)> 0$ independent of $h >0$. Indeed, we simply note that $\Omega(k) = 4 \sin^2(k/2) \geq \frac{4}{\pi^2} k^2$ for $|k| \leq \pi$. Hence first of the desired inequalities follows from the fact that $(1+|z|^{2\sigma}) \leq C (1+ z^2)$ for all $z \in \R$, with some constant $C=C(\sigma) >0$, provided that $0 \leq \sigma \leq 1$ holds. Moreover,  is easy to see that we have $(1+h^{-2} \Omega(k)) \geq C (1+ h^{-2} k^2)$ for $|k| \leq \pi$ and $h >0$. This shows the second inequality stated in Lemma \ref{lem:hfull}.  \end{proof}

\subsection{Interpolations and norm estimates}
In this subsection we collect some technical results about the discretization and interpolation of functions. More precisely, for a locally integrable function $f : \R \to \C$, we recall that its discretization $f_h : h \Z \to \C$ is given by
$$
f_h(x_m) = \frac{1}{h} \int_{x_m}^{x_{m+1}} f(x) \, dx, \quad \mbox{with $x_m=hm$ and $m \in \Z$}.
$$
Following \cite{L}, we define piecewise constant interpolation  $q_h f_h$ by
\begin{equation} \label{def:qh}
(q_h f_h)(x) := f_h(x_m), \quad \mbox{for $x \in [x_m, x_{m+1})$.}
\end{equation}
Furthermore, we recall its piecewise linear interpolation $p_h f_h$ introduced in \eqref{def:ph}. We begin with the following simple fact.

\begin{lemma} \label{prop:unifH2}
For any $0 \leq \sigma \leq 1$ and $f \in H^s(\R)$, we have that
$$
\| f_h \|_{H^\sigma_h} \leq C \| f \|_{H^{\sigma}},
$$
where the constant $C > 0$ is independent of $h> 0$ and $f$.
\end{lemma}

\begin{proof}
First, we note that, by standard interpolation theory, we have (with norm equivalences) that 
$$
H^{\sigma}(\R) = \left [L^2(\R), H^1(\R) \right ]_{\sigma,2} , \quad H^\sigma_h = \left [ L^2_h, H^1_h \right ]_{\sigma,2}, \quad \mbox{for $0 \leq \sigma \leq 1$}.
$$
Hence, it suffices to prove the claimed bound for the endpoint cases $s=0$ and $s=1$. Indeed, by the Cauchy--Schwarz inequality,
\begin{align*}
\| f_h \|_{L^2_h} & = h \sum_m \frac{1}{h^2} \left | \int_{mh}^{(m+1) h} f(x) \,dx \right |^2 \leq  \sum_m \int_{mh}^{(m+1)h} |f(x)|^2 \, dx = \| f \|_2^2.
\end{align*}
To deal with the case when $s=1$, we note that, by the generalized mean-value theorem and the Cauchy-Schwarz inequality, we have
$$
\int_{I} \left | f(x+h)- f(x) \right |^2 \,dx \leq h^2 \int_I \int_0^1 |f'(x+th)|^2 \, dt \, dx,
$$
for any interval $I \subset \R$.  Using this bound, the Cauchy-Schwarz inequality again and Fubini's theorem, we deduce that
\begin{align*}
\| D_h^+ f_h \|_{L^2_h}^2 & = h \sum_m \frac{1}{h^4} \left | \int_{mh}^{(m+1)h} \left ( f(x+h)-f(x) \right ) \,dx \right |^2 \\
& \leq h \sum_m \frac{1}{h^4} \cdot h \cdot h^2 \cdot  \int_{mh}^{(m+1)h} \int_0^1 \left | f'(x+th) \right |^2 \, dt \, dx \\
& = \int_0^1 \sum_m  \int_{mh}^{(m+1)h} |f'(x+th)|^2 \,  dx \, dt = \int_0^1 \| f' \|_{L^2}^2 \, dt = \| f' \|_{L^2}^2,
\end{align*}
which shows that $\| D^+_h f_h \|_{L^2_h} \leq  \| f' \|_{L^2}$. Finally, with help of Lemma \ref{lem:hfull}, we obtain
$$
\| f_h \|_{H^1_h} \leq C \| f_h \|_{\tilde H^1_h} \leq C \| f \|_{H^1},
$$
with $C > 0$ independent of $h> 0$ and $f$. The proof of Lemma \ref{prop:unifH2} is now complete. \end{proof}

As a next result, we derive uniform estimates with respect to (fractional) Sobolev norms.

\begin{lemma} \label{lem:interbound}
For all $0 \leq \sigma \leq 1$, we have the bounds
$$
\| q_h f_h \|_2 \leq \| f_h \|_{L^2_h} , \quad \| p_h f_h \|_{H^\sigma} \leq C  \| f_h \|_{H^\sigma_h},
$$
with some constant $C > 0$ independent of $h > 0$ and $f$.
\end{lemma}

\begin{proof}
The bound for $q_h f_h$ follows readily from
$$
\| q_h f_h \|_2^2 = \sum_m \int_{mh}^{(m+1) h} |f_h(mh)|^2 \, dx = h \sum_m |f_h(x_m)|^2 = \| f_h \|_{L^2_h}^2 .
$$
To prove the claimed bounds for $p_h f_h$, we argue as follows. As in the proof of Lemma \ref{prop:unifH2}, we can use interpolation of norms to conclude that it suffices to prove the bounds for $\sigma=0$ and $\sigma=1$, i.\,e.,
\begin{equation} \label{ineq:sobo}
\| p_h f_h \|_2 \leq C \| f_h \|_{L^2_h}, \quad \| p_h f_h \|_{H^1} \leq C \| f_h \|_{H^1_h} ,
\end{equation}
with $C >0$ independent of $h> 0$ and $f$. Indeed, we note that
$$
p_h f_h (x) = q_h f_h(x) + \sum_m (D^+_h f_h)(x_m) \1_{[x_m, x_{m+1})}(x) (x-x_m) .
$$ 
Hence, using that $\| q_h f_h \| \leq \| f_h\|_{L^2_h}$, we find that
$$
\| p_h f_h \|_2 \leq \| q_h f_h \|_2 + A \leq \| f_h \|_{L^2_h} + A,
$$
where
\begin{align*}
A^2 & = \sum_{m} \int_{x_m}^{x_{m+1}} | D^+_h f_h(x_m) |^2 (x-x_m)^2 \, dx= \sum_m \frac{h^3}{3} | D^+_h f_h(x_m)|^2 \\
& = \frac{h}{3} \sum_m \left | f_h(x_{m+1}) - f_h(x_m) \right |^2 \leq \frac{2 h}{3} \sum_m |f_h(x_m)|^2 = \frac{2}{3} \| f_h \|_{L^2_h}^2.
\end{align*}
Thus we obtain that $\| p_h f_h \|_2 \leq C \| f_h \|_{L^2_h}$, which proves the first bound in \eqref{ineq:sobo}. 

To show the second bound in \eqref{ineq:sobo} we argue as follows. From \cite[Chapter VI]{L}, we recall that 
$$\frac{d}{dx} p_h u_h= q_h(D^+_h u_h),$$
 where $q_hf_h$ is the piecewise constant interpolation of $f_h$ defined  in \eqref{def:qh}. Using the previous bounds, we have
\begin{align*}
\| p_h f_h \|_{H^1}  & \leq \| p_h f_h \|_{L^2} + \| q_h(D^+_h f_h) \|_{L^2}  \\ & \leq C ( \| f_h \|_{L^2_h} + \| D^+_h f_h \|_{L^2_h} ) \leq C \| f_h \|_{\tilde{H}^1_h} \leq C \| f_h \|_{H^1_h} , \nonumber
\end{align*}
where we used Lemma \ref{lem:hfull} in the last inequality. This completes the proof of the second bound in \eqref{ineq:sobo}, and hence Lemma \ref{lem:interbound} is proven. \end{proof}

We conclude this subsection with a technical fact that will be used below.

\begin{lemma} \label{lem:convph}
For any $f \in L^2(\R)$ and $g \in H^1(\R)$, we have
$$
\mbox{$\| p_h f_h -f \|_2 \to 0$ and $\| p_h g_h - g\|_{H^1} \to 0$ as $h \to 0^+$}.
$$
Moreover, if $v \in H^{\sigma}(\R)$ for some $0 < \sigma < 1$, then $p_h v_h \weakto v$ weakly in $H^\sigma(\R)$ as $h \to 0^+$. 
\end{lemma}

\begin{proof} For the strong convergence results in $L^2(\R)$ and $H^1(\R)$, we refer to \cite[Chapter VI, Lemma 4.1]{L}. 

The weak convergence result can be seen as follows. Suppose that $v \in H^{\sigma}(\R)$ for some $0 < \sigma < 1$. Then $p_h v_h \to v$ strongly in $L^2(\R)$ as $h \to 0^+$. Let $h_n \to 0^+$ be some sequence. By Lemma \ref{prop:unifH2} and \ref{lem:interbound}, we have that $\sup_{n \geq 1} \|  p_{h_n} v_{h_n} \|_{H^\sigma} < +\infty$. Hence, after passing to a subsequence if necessary, we can assume that $p_{h_n} v_{h_n} \weakto w$ weakly in $H^\sigma(\R)$ as $n \to \infty$ for some $w \in H^\sigma(\R)$. But since $p_h v_h \to v$ strongly in $L^2(\R)$ as $h \to 0^+$, we conclude by a density argument that $w \equiv v$ holds. Hence $p_h v_h \weakto v$ weakly in $H^\sigma(\R)$ as $h \to 0^+$.   \end{proof}

\subsection{Strong Convergence for $\Ldif^{J}_h$ as $h \to 0^+$}

In this subsection, we study the ``continuum limit'' of the operator $\Ldif^{J}_h$ defined in \eqref{def:Jop}. To this end, we extend the action of $\Ldif^{J}_h$ to functions $\phi \in L^2(\R)$ by setting
\begin{equation}\label{Lj}
 (\Ldif^{J}_h \phi)(x) := \frac{1}{\beta(h)} \sum_{n \neq 0} J_{|n|} \big [  \phi(x) - \phi(x-nh) \big ]  
\end{equation}
where $\beta(h)$ is given in \eqref{eq:beta}.  

We readily check that $\| \Ldif^{J}_h \phi \|_2 \leq  C \beta(h)^{-1} \sum_{n =1}^\infty J_n \| \phi \|_2 \leq C \beta(h)^{-1} \| \phi \|_2$, and thus the operator $\Ldif^{J}_h$ is bounded on $L^2(\R)$. Moreover, we easily check that $( \Ldif^{J}_h)^* = \Ldif^{J}_h$ is self-adjoint. Furthermore, recalling the averaged discretization of $v \in L^2(\R)$, which is $v_h(x_m) = h^{-1} \int_{x_m}^{x_{m+1}} v(x) \, dx,$ we see that $\| v_h \|_{L^2_h} \leq \| v \|_2$ by Lemma \ref{prop:unifH2}. Moreover, a simple calculation shows that 
\begin{equation} \label{eq:Lcommute}
\left  (\Ldif^{J}_h v \right )_h (x_m) = \left (\Ldif^{J}_h v_h \right )(x_m) .
\end{equation} 
This identity says that we can first let $\Ldif^{J}_h$ act on $v \in L^2(\R)$ and then discretize, or equivalently first discretize $v$ and let the discrete operator $\Ldif^{J}_h$ act on $v_h \in L^2_h$. This fact will be needed further below when we discuss the continuum limit.

We conclude this section with the following convergence result. 

\begin{lemma} \label{lem:conv}
Let $0 < s \leq +\infty$ and suppose that $J = (J_n )_{n=1}^\infty \in \mathcal{K}_s$ with $J \not \equiv 0$. Define $\alpha$ as in \eqref{alpha}. Then, for every $\phi \in C^\infty_0(\R)$, 
$$
\mbox{$\Ldif^{J}_h \phi \to c (-\Delta)^{\alpha} \phi$ strongly in $L^2(\R)$ as $h \to 0^+$} .
$$
Here $c > 0$ is some constant that only depends on $s$ and $J$. 
\end{lemma} 

\begin{remark} \em
Putting it differently, this lemma says that the family of bounded self-adjoint operators $\Ldif^{J}_h$ converges strongly  as $h \to 0^+$ to the unbounded self-adjoint operator $c (-\Delta)^{\alpha}$ acting on $L^2(\R)$ with dense domain $C^\infty_0(\R)$.
\end{remark}

\begin{proof}
By taking the Fourier transform $\mathcal{F}$ on $\R$, we find that
$$
\mathcal{F}(\Ldif^{J}_h \phi)(k) =  \frac{ \omega(hk)}{  \beta (h) |k|^{2 \alpha}}  |k|^{2 \alpha} \hat{\phi}(k), \quad \mbox{for a.\,e.~$k \in \R$},
$$
where $\hat{\phi}=\mathcal{F}(\phi)$ and $\omega(z) := 2 \sum_{n=1}^\infty J_n [ 1- \cos (nz) ]$ with $z \in \R$. Let $k \in \R$ with $k \neq 0$ be fixed. If $s \neq 1$, then $\beta(h) = h^{2\alpha}$ and hence, by Lemma \ref{lem:asymp}, 
$$
\lim_{h \to 0} \frac{\omega(hk)}{\beta(h) |k|^{2 \alpha}} = \lim_{z \to 0} \frac{\omega(z)}{|z|^{2\alpha}} = c,
$$ 
for some constant $c >0$. If $s=1$, we have $\beta(h) = -\log(h) h^2$. Letting $z= hk$ with $k \neq 0$ fixed, we conclude in this case, by using Lemma \ref{lem:asymp} again, 
$$
\lim_{h \to 0} \frac{\omega(hk)}{\beta(h) |k|^2} = - \lim_{z \to 0} \frac{ \omega(z) }{ \log(z/k) |z|^2} = c,
$$
for some constant $c > 0$. In summary, we conclude that for any $0 < s \leq +\infty$ the pointwise convergence
$$
\mbox{$\mathcal{F}(\Ldif^{J}_h \phi)(k) \to c |k|^{2 \alpha} \hat{\phi}(k)$ for a.\,e.~$k \in \R$ as $h \to 0^+$.}
$$
To turn this into strong convergence in $L^2(\R)$, we derive bounds uniform in $h>0$ and then use the dominated convergence theorem. First, we assume that $s \neq 1$ and hence $\beta(h) = h^{2 \alpha}$. In this case, we have, by Lemma \ref{lem:asymp},
$$
\left | \frac{\omega(hk)}{|hk|^{2 \alpha}} \right | \leq 2 c ,\quad \mbox{for $|hk| \leq z_0$},
$$
where $z_0 > 0$ is some small constant depending only on $J$. On the other hand, we have the upper bound $|\omega(hk)| \leq 4 \sum_{n=1}^\infty J_n \leq A$ for some $A > 0$. Therefore,
$$
\left | \frac{\omega(hk)}{|hk|^{2 \alpha}} \right | \leq \frac{A}{|z_0|^{\beta}} \leq C, \quad \mbox{for $|hk| \geq z_0$}.
$$
Combining these bounds, we conclude that $| \frac{\omega(hk)}{\beta(h) |k|^{2\alpha}} | \leq C$ for all $h > 0$ and $k \in \R$. Hence,
$$
| \mathcal{F} (\Ldif^{J}_h \phi)(k)| \leq C |k|^{2 \alpha} |\hat{\phi}(k)| , \quad \mbox{for a.\,e.~$k \in \R$},
$$
with some constant $C > 0$ depending only on $s$ and $J$, provided that $s \neq 1$ holds. Note that $\int |k|^{2 \alpha} |\hat{\phi}(k)|^2 < +\infty$, since $\phi$ belongs to $C^\infty_0(\R)$. By the dominated convergence theorem, we deduce that
$$
 \int_\R \Big | \mathcal{F}(\Ldif^{J}_h  \phi)(k) - c  |k|^{2 \alpha} \hat{\phi}(k) \Big |^2 \, d k \to 0 \quad \mbox{as} \quad h \to 0^+,
$$
which completes the proof of Lemma \ref{lem:conv}, provided that $s \neq 1$ holds.

To complete the proof for the special case $s=1$, we note that $\beta(h) = (- \log h) h^2$. Since $|\frac{\omega(z)}{\log (z) z^2} | \leq C$ for $|z| \leq z_0$ with some constant $z_0 > 0$ by Lemma \ref{lem:asymp}, we deduce, for $|hk| \leq z_0$, that
$$
\left | \frac{\omega(hk)}{\beta(h) |k|^2 } \right | = \left | \frac{ \omega(hk) \log(hk)}{\log (h) \log (hk) |h^2 k^2|} \right | \leq C \frac{\left | \log(hk) \right | }{\left | \log (h) \right | } \leq \frac{C}{\left |\log h_0 \right|} \left (1 + \left |\log k \right | \right ) ,
$$
where we assume without loss of generality that $0 < h \leq h_0 < 1$. On the other hand, using that $|\omega(hk)| \leq 4 \sum_n J_n \leq C$, we find that
$$
\left | \frac{\omega(hk)}{\beta(h) |k|^2 } \right | \leq \frac{C}{\left | \log(h) h^2 k^2 \right | } \leq \frac{C}{\left | \log (h_0) z_0^2 \right |}, \quad \mbox{for $|h k| \geq z_0$},
$$
and $0 < h \leq h_0 < 1$. In summary, we have shown in the case $s=1$ that 
$$
| \mathcal{F} (\Ldif^{J}_h \phi)(k)| \leq C (1 +  \left | \log k \right |) |k|^{2} |\hat{\phi}(k)| , \quad \mbox{for a.\,e.~$k \in \R$},
$$
with some constant $C > 0$ depending only on $J$ and $h_0 > 0$. We easily check that the integral $\int (1 + \left | \log k \right |)^2 |k|^4 |\hat{\phi}(k)|^2 \, dk$ is finite, since $\phi \in C^\infty_0(\R)$. By dominated convergence, we deduce that Lemma \ref{lem:conv} also holds when $s=1$. \end{proof}

\section{Discrete Evolution Problem and A-Priori Bounds}\label{discretevolution}

\label{sec:evolution}

Suppose that $J = ( J_n )_{n=1}^\infty \in \mathcal{K}_s$ for some $0 < s \leq +\infty$. Let $\Ldif^{J}_h$ be the discrete operator defined \eqref{Lj} above. We consider the initial-value problem 
\begin{equation} \label{eq:d_ivp}
\left \{ \begin{array}{l} \displaystyle i \frac{d}{dt} u_h(t,x_m) =  (\Ldif^{J}_h u_h)(t,x_m) \pm |u(t,x_m)|^2 u(t,x_m) , \\[1ex]
u_h(0,x_m) = v_h(x_m) \quad \mbox{with $m \in \Z$}. \end{array} \right . 
\end{equation}
 We record the following simple fact.
\begin{prop}[GWP in $L^2_h$] \label{prop:gwp_discrete}
The initial-value problem \eqref{eq:d_ivp} is globally well-posed in $L^2_h$. That is, for every initial datum $v_h \in L^2_h$, there exists a unique global classical solution $u_h \in C^1([0,\infty); L^2_h)$ that solves \eqref{eq:d_ivp}. 

Moreover, we have conservation of energy and $L^2_h$-mass given by
$$
E_h(u_h) = \frac{1}{2} (u_h, \Ldif^{J}_h u_h)_{L^2_h} \pm \frac{1}{4} \| u_h \|_{L^4_h}^4,
$$
$$
N(u_h) = (u_h, u_h)_{L^2_h} .
$$
\end{prop}

\begin{proof}
This  follows from standard arguments. Indeed, we consider the integral formula
$$
u_h(t) = e^{-i t \Ldif^{J}_h} v \mp i \int_0^t e^{-i(t-s) \Ldif_h^{J}} |u_h(s)|^2 u_h(s) \, ds.
$$
Note that $\{ e^{-it \Ldif^{J}_h} \}_{t \in \R}$ is strongly continuous unitary one-parameter group  on $L^2_h$, since $\Ldif^{J}_h$ is self-adjoint. Moreover, note that $u_h \mapsto |u_h|^2 u_h$ is locally Lipschitz on $L^2_h$ thanks to the embedding $\| \cdot \|_{\ell_\infty} \leq \| \cdot \|_{\ell_2}$ for sequences. A simple fixed point argument now yields local well-posedness in $L^2_h$, where the local time of existence only depends on $\| u_h \|_{L^2_h}$. Global extension then follows from $L^2_h$-conservation. Finally, note that any  solution $u_h \in C^0(\R; L^2_h)$ is automatically a strong classical solution, since $\frac{d}{dt} u_h \in L^2_h$ by the equation itself, since $\Ldif^{J}_h u_h \in L^2_h$ (because $\Ldif^{J}_h$ is a bounded operator on $L^2_h$) and $|u_h|^2 u_h \in L^2_h$ as previously remarked.  

The proof of the conservation laws follows from a simple calculation.\end{proof}

\medskip
Next, we derive the following a priori bounds for solutions $u_h$ to \eqref{eq:d_ivp}. 

\begin{lemma}[A priori bounds] \label{lem:bounds}
Let $u_h \in C^1([0,\infty); L^2_h)$ solve \eqref{eq:d_ivp}. Suppose that $J = (J_n)_{n=1}^\infty$ belongs to $\mathcal{K}_s$ with some $\frac{1}{2} < s \leq +\infty$ and assume that $ J_1 >0$. Define $\sigma=s$ if $0 < s < 1$, $\sigma=1-\eps$ for some $\eps > 0$ small if $s=1$, and $\sigma = 1$ if $s > 1$. Then if $s>0 $ and $s\ne 1$ we have the uniform a-priori bound
$$
\sup_{t \geq 0} \| u_h(t, \cdot) \|_{H^{\sigma}_h} \leq C ( \| v_h \|_{H^{\sigma}_h} ) .
$$ 
If $s=1$, for any $0 < T < \infty$, we have the bound
$$
\sup_{t \in [0,T]} \| u_h(t, \cdot) \|_{ H^1_h} \leq C ( T, \| v_h \|_{ H^1_h} ) .
$$
Here the constants $C > 0$ are independent of $0 < h \leq h_0$ with $h_0 > 0$ sufficiently small.
\end{lemma} 

%\begin{remark}
%In fact, the control of $\| u_h(t, \cdot) \|_{\tilde{H}^1_h}$ on compact time intervals will only be needed to treat the borderline case when $s=1$ in the proof of Theorem \ref{thm:main} below.
%\end{remark}

\begin{proof}
The global a-priori bounds follow from conservation of $E_h(u_h)$ and $N_h(u_h)$ and Lemmas \ref{lem:GN} and \ref{lem:embed}. Indeed, in the defocusing case (with $+$ sign), we immediately obtain the a-priori bound 
$$
\| u_h(t, \cdot ) \|_{H^\sigma_h}^2 \leq C \| u_h(t,\cdot) \|_{H^{J}_h}^2 \leq C (E(v_h) + N(v_h)) \leq C ( \| v_h \|_{H^{\sigma}_h} ),
$$
where we used Lemmas \ref{lem:GN} and \ref{lem:embed}. In the focusing case, we use Lemmas \ref{lem:GN} and \ref{lem:embed} and $(u_h, u_h)_{L^2_h} = \mbox{const}$., to deduce that
$$
E(u_h) \geq  \frac{1}{2}  \| u_h \|_{H^{\sigma}_h}^2 - C \| u_h \|_{H^{\sigma}_h}^{4\sigma_0/\sigma}  
$$
for any fixed $1 \geq \sigma_0 > \frac 1 4$. Since $\sigma > \frac 1 2$ by assumption, we can ensure that $\frac {4\sigma_0} \sigma < 2$ holds. From this bound we deduce that $\| u_h \|_{H^{\sigma}_h} \to +\infty$ implies that $E(u_h) \to +\infty$. Hence, by energy conservation and $E(u_h) <+\infty$, we can deduce the a-priori bound
$$
\sup_{t \geq 0} \| u_h(t,\cdot) \|_{H^{\sigma}_h} \leq C ( \| v_h \|_{H^{\sigma}_h} ) .
$$

Now we turn to the a-priori bound for the integer discrete (classical) Sobolev norm $\| u_h(t,\cdot) \|_{H^1_h}$ defined in \eqref{classicaldiscrete}. By Lemma \ref{lem:hfull}, we can use the equivalent norm $\| \cdot \|_{\tilde H^1_h}$ instead of $\| \cdot \|_{H^1_h}$.

By Duhamel's formula,
$$
u_h(t) = e^{-it \Ldif^{J}_h} v_h \mp i \int_0^t e^{-i(t-s) \Ldif^{J}_h} |u_h(s)|^2 u_h(s) \, ds,
$$
where we write $u_h(t)$ to denote $u_h(t,x_m)$ with $m \in \Z$, etc. Taking the discrete right-hand derivative $D^+_h$, we obtain that
\begin{align*}
\| D^+_h u_h(t) \|_{L^2_h} & \leq \| D^{+}_h e^{-it \Ldif^{J}_h } v_h \|_{L^2_h} + \int_0^t \| D^+_h \{ e^{-i(t-s) \Ldif^{J}_h} |u_h(s)|^2 u_h(s) \} \|_{L^2_h} \, ds.
\end{align*}
We have that $D^+_h e^{-it \Ldif^{J, \beta}_h} = e^{-it \Ldif^{J}_h} D^+_h$, since $\Ldif^{J}_h$ commutes with translations on $h\Z$. Moreover, recall that $e^{-it \Ldif^{J}_h}$ is unitary on $L^2_h$. Hence,
$$
\| D^+_h u_h(t) \|_{L^2_h} \leq \| D^+_h v_h \|_{L^2_h} + \int_0^t \| D^+_h \{ |u_h(s)|^2 u_h(s)\} \|_{L^2_h} \, ds. 
$$
By the Leibniz product formula for $D^+_h$,
$$ 
D^+_h (u_h v_h) (s) =  u_h (s) D^+_h v_h (s) + D^+_h u_h (s) v_h (s+h),
$$
 and the uniform embedding Lemma \ref{lem:sobo} for $\sigma > \frac{1}{2}$, we deduce that
$$
\| D^{+}_h \{ |u_h(s)|^2 u_h(s)\} \|_{L^2_h}  \leq C \| u_h \|_{L^\infty_h}^2 \| D^+_h u_h(s) \|_{L^2_h} \leq C(\| v_h\|_{H^{\sigma}_h}) \| D^+_h u_h(s) \|_{L^2_h},
$$
where we also used Lemma \ref{lem:embed} and the a-priori bound on $\| u_v \|_{H^{\sigma}_h}$ derived above. In summary, we see that
$$
 \| D^+_h u_h(t) \|_{L^2_h} \leq \| D^+_h v_h \|_{L^2_h} + C  \int_0^t \| D^{+}_h u_h(s) \|_{L^2_h} \, ds. 
$$
By Gronwall's estimate, this implies that
$$
 \sup_{t \in [0,T]}  \| D^+_h u_h(t) \|_{L^2_h} \leq C(T, \| v_h \|_{\tilde H^1_h}, \| v_h \|_{H^{\sigma}_h} ),
$$
for any fixed $T >0$. Noting that $\|v_h \|_{H^{\sigma}_h} \leq C \| v_h \|_{H^1_h}$ by Lemma \ref{lem:hfull}, we complete the proof of the desired a-priori bound for $\| u_h \|_{\tilde{H}^1_h}$. The proof of Lemma \ref{lem:bounds} is now complete.
\end{proof}

\section{Proof of Theorem \ref{thm:main}}

\label{sec:proof}

%{\bf I will focus on the case $\frac{1}{ 2}  < s <1$ first... The case $s > 1$ follows similarly. To be done later. Again, the borderline case $s=1$ (yielding $\log$-terms) is technically a bit cumbersome... So I will omit it for the time being.}

For the reader's convenience, we first recall the hypotheses and definitions from Theorem \ref{thm:main}. We suppose that $J = (J_n) \in \mathcal{K}_s$ for some $\frac 1 2 < s \leq +\infty$ and $J_1 > 0$. Let 
$$
\alpha := \left \{ \begin{array}{ll} s & \quad \mbox{if $\frac 1 2 < s < 1$,} \\ 1 & \quad \mbox{if $s \geq 1$}. \end{array} \right .
$$
Let $\beta(h)$ as in Theorem \ref{thm:main}. Assume that $0< h_0 < 1$ is sufficiently small and consider the lattice $h \Z$ with $0 < h \leq h_0$. Suppose that $v \in H^\alpha(\R)$ and let $v_h=v_h(x_m)$ be its discretization defined in \eqref{eq:discretize}. Note that by Lemma \ref{lem:convph} we have $p_h v_h \weakto v$ weakly in $H^\alpha(\R)$ as $h \to 0^+$, where $p_h$ is the piecewise linear interpolation defined in \eqref{def:ph} above. Finally, let $u_h= u_h(t,x_m)$ denote the corresponding global solution to the discrete evolution problem \eqref{eq:discivp} with initial datum $v_h \in L^2_h$.  

Let $T > 0$ be a fixed (but otherwise arbitrary) time. As a first step in the proof of Theorem \ref{thm:main}, we derive the following uniform bounds for $p_h u_h(t)$ and $\partial_t p_h u_h(t)$. 

\subsubsection*{\bf Bounds for $p_h u_h(t)$ in $H^\alpha(\R)$} Let $M:=\sup_{0 < h \leq h_0} \| p_h v_h \|_{H^\alpha}$. Note that $M < +\infty$ for $h_0 > 0$ sufficiently small, since $p_h v_h \weakto v$ weakly in $H^\alpha(\R)$ as $h \to 0^+$ as mentioned above. Now, we claim that
\begin{equation} \label{ineq:discbound}
\sup_{t \in [0,T]} \| p_h u_h(t) \|_{H^\alpha} \leq C,
\end{equation}
where $C >0$ is some constant that only depends on $h_0, s, T$ and $M$. To prove \eqref{ineq:discbound}, we first recall from Lemma \ref{lem:interbound} that
$$\| p_h u_h(t) \|_{H^\alpha} \leq C\|u_h(t)\|_{H^\alpha_h}.$$
Next, by Lemma \ref{lem:bounds}, we have the a-priori bound on $[0,T]$ given by\footnote{To be precise the time dependence on $T$ only appears when $s=1$.}
$$\sup_{t \in [0,T] } \| u_h(t) \|_{H^\alpha_h} \leq C (T, \| v_h \|_{H^\alpha_h}).$$
Finally, we have the general bound $\| v_h \|_{H^\alpha_h} \leq C \| v \|_{H^\alpha}$ by Lemma \ref{prop:unifH2} with $C >0$ independent of $h > 0$. Hence, we deduce that \eqref{ineq:discbound} holds.

\subsubsection*{\bf Bounds for $\partial_t p_h u_h(t)$ in $H^{-\alpha}(\R)$} We claim that
\begin{equation} \label{ineq:discbound2}
\sup_{t \in [0,T]} \| \partial_t p_h u_h(t) \|_{H^{-\alpha}} \leq C ,
\end{equation}
with some constant $C > 0$ that only depends on $h_0, s, T$ and $M = \sup_{0 < h \leq h_0} \| p_h v_h \|_{H^\alpha} < +\infty$. Indeed, from \eqref{eq:discivp} we obtain the estimate
$$
\| \partial_t u_h(t) \|_{H^{-\alpha}_h} \leq \| \Ldif^{J}_h u_h(t) \|_{H^{-\alpha}_h} + \| | u_h(t) |^2 u_h(t) \|_{H^{-\alpha}_h} ,
$$
where we refer to the definition of the dual norm $\| \cdot \|_{H^{-\alpha}_h}$ in Appendix \ref{dual} below. By Proposition \ref{prop:dualbound} and the previous bounds, we conclude
$$
\| \Ldif^{J}_h u_h(t) \|_{H^{-\alpha}_h} \leq C \sup_{t \in [0,T]} \| u_h(t) \|_{H^\alpha_h} \leq C(T, \| v_h \|_{H^\alpha_h} ) .
$$
Furthermore, we deduce that
$$
\| | u_h(t) |^2 u_h(t) \|_{H^{-\alpha}_h}  \leq \| | u_h(t) |^2 u_h(t) \|_{H^{\alpha}_h} \leq C(T,   \| v_h \|_{H^\alpha_h} ) ,
$$
where in the last step we used the Leibniz rule, the uniform embedding Lemma \ref{lem:sobo}, and again the uniform bound on $\| u_h(t) \|_{H^\alpha_h}$ on $[0,T]$. Hence, we have shown that 
$$
\sup_{t \in [0,T]} \| \partial_t u_h(t) \|_{H^{-\alpha}_h} \leq C
$$
for some constant $C > 0$ that only depends on $h_0, s, T$ and $M=\sup_{0 < h \leq h_0} \| p_h v_h \|_{H^\alpha}$. From Proposition \ref{prop:dual_ph} we have
$$
\| p_h f_h \|_{H^{-\alpha}} \leq C \| f_h \|_{H^{-\alpha}_h},
$$
with some constant $C > 0$ independent of $h> 0$, and the fact that $p_h$ commutes with $\partial_t$, we deduce that \eqref{ineq:discbound2} holds. 

\subsubsection*{\bf Weak-$*$ compactness} 
By the uniform bounds \eqref{ineq:discbound} and \eqref{ineq:discbound2}, we deduce by the Banach-Alaoglu theorem that
\begin{equation} \label{eq:weak1}
\mbox{$p_{h_n} u_{h_n} \weakto u$ weakly-$*$ in $L^\infty([0,T]; H^\alpha(\R))$ as $n \to \infty$},
\end{equation}
\begin{equation} \label{eq:weak2}
\mbox{$\partial_t p_{h_n} u_{h_n} \weakto \partial_t u$ weakly-$*$ in $L^\infty([0,T]; H^{-\alpha}(\R))$ as $n \to \infty$},
\end{equation}
with some sequence $h_n \to 0$ as $n \to \infty$. Note that, by standard arguments, the fact that $u \in L^\infty ([0,T]; H^\alpha(\R)) \cap W^{1,\infty}([0,T]; H^{-\alpha}(\R))$ implies that $u \in C^0([0,T]; L^2(\R))$. In particular, the notion of an initial condition for $u(0)$ is well-defined. Next, we recall that $p_h u_h(0) \weakto v$ weakly in $H^\alpha(\R)$ by Lemma \ref{lem:convph}, we deduce that $u \in C^0([0,T]; L^2(\R))$ satisfies 
$$
u(0) = v \in H^\alpha(\R).
$$ 
Next, we claim that the limit $u=u(t,x)$ solves the following initial-value problem:
\begin{equation} \label{eq:limitivp}
 \left \{ \begin{array}{l} u \in L^\infty([0,T]; H^{\alpha}(\R)) \cap W^{1,\infty}([0,T]; H^{-\alpha}(\R)) , \\
 i \partial_t u = c (-\Delta)^{\alpha} u \pm |u|^2 u, \quad \mbox{for a.\,e.~$t \in [0,T]$} ,\\
 u(0) = v \in H^\alpha(\R) . \end{array} \right . 
\end{equation}
Here $c > 0$ is some suitable constant chosen below. 

\begin{prop} \label{prop:wp2}
Let $T >0$ be given and suppose that $u= u(t,x)$ solves \eqref{eq:limitivp}. Then $u \in C^0([0,T]; H^\alpha(\R))$ holds and $u=u(t,x)$ is the unique solution given by Proposition \ref{NLScont} above.
\end{prop}

\begin{proof}
From standard theory for abstract evolution equations, we deduce that $u=u(t,x)$ solving \eqref{eq:limitivp} satisfies the integral equation
$$
u(t) = e^{-it(-\Delta)^{\alpha} } v - i \int_0^t e^{-i(t-s) (-\Delta)^{\alpha}} |u(s)|^2 u(s) \, ds.
$$
Note that the map $u \mapsto |u|^2 u$ is locally Lipschitz on $H^\alpha(\R) \subset L^\infty(\R)$ when $s > \frac 1 2$. Hence we deduce that $u \in C^0([0,T]; H^\alpha(\R)$ and we can now apply Proposition \ref{NLScont}.
\end{proof}

To conclude the proof that the limit $u=u(t,x)$ in \eqref{eq:weak1} and \eqref{eq:weak2} is the unique solution of \eqref{eq:limitivp}, it remains to show that 
\begin{equation} \label{eq:ivp_ae_t}
i \partial_t u = c (-\Delta)^{\alpha} u \pm |u|^2 u, \quad \mbox{for a.\,e.~$t \in [0,T]$} ,
\end{equation}
where $c > 0$ is some constant. Note that, by \eqref{eq:weak2}, we directly have that
\begin{equation}
\int_0^T \langle \Phi, i \partial_t p_{h_n} u_{h_n}(t) \rangle \, dt \to \int_0^T \langle \Phi, i \partial_t u(t) \rangle \, dt \quad \mbox{as} \quad n \to \infty,
\end{equation}
for every $\Phi \in L^1([0,T]; H^{\alpha}(\R))$, where $\langle \cdot, \cdot \rangle$ denotes the usual inner product on $L^2(\R)$.

Next, we claim that 
\begin{equation} \label{eq:weakconvmaster}
\int_0^T \left \langle \Phi, p_{h_n} \Ldif^{J}_{h_n} u_{h_n}(t) \right \rangle \, dt \to \int_0^T \left \langle \Phi, c (-\Delta)^{\alpha} u(t) \right \rangle \, dt \quad \mbox{as} \quad n \to \infty,
\end{equation}
for every $\Phi \in L^1([0,T]; H^{\alpha}(\R))$. By a density argument, it suffices to prove this claim for $\Phi(t,x) = f(t) w(x)$ with $f \in C^\infty_0([0,T])$ and $w \in C^\infty_0(\R)$. For $h > 0$, we define the usual discretization $w_h \in L^2_h$ by setting
$$
w_h(x_m) = \frac{1}{h} \int_{x_m}^{x_{m+1}} w(x) \, dx.
$$
Since $\alpha \leq 1$ and  $w \in C^\infty_0(\R) \subset H^1(\R)$, we can apply Lemma \ref{lem:convph} to conclude 
$$
\| p_h w_h - w \|_{H^\alpha} \leq C \| p_h w_h - w \|_{H^1} \to 0 \quad \mbox{as} \quad  h \to 0^+.
$$ 
 Furthermore, recall the uniform bound $\| p_h \Ldif^{J}_{h} u_h(t) \|_{H^{-\alpha}} \leq C$. Hence
$$
\langle p_{h_n} w_{h_n} -w, p_{h_n} \Ldif^{J}_{h_n} u_{h_n}(t) \rangle \leq C \| p_{h_n} w_{h_n} - w \|_{H^\alpha} \to 0 \quad \mbox{as} \quad n \to \infty,
$$
for every $t \in [0,T]$. Thus it suffices to show that
\begin{equation} \label{eq:weakcv}
\int_0^T \left \langle f p_{h_n} w_{h_n}, p_{h_n} \Ldif^{J}_{h_n} u_{h_n}(t) \right \rangle \, dt \to \int_0^T \left \langle f w, c (-\Delta)^{\alpha} u(t) \right \rangle \, dt \quad \mbox{as} \quad n \to \infty,
\end{equation} 
for every $f = f(t) \in C^\infty_0([0,T])$ and $w = w(x) \in C^\infty_0(\R)$. Next, by Lemma \ref{lem:discintpart}, we see that
$$
\left \langle  p_{h_n} w_{h_n}, p_{h_n} \Ldif^{J}_{h_n} u_{h_n}(t) \right \rangle = \left \langle p_{h_n} \Ldif^{J}_{h_n} w_{h_n}, p_{h_n} u_{h_n}(t) \right \rangle \quad \mbox{for every $t \in [0,T]$}.
$$
From \eqref{eq:weak1} we can assume that $p_{h_n} u_{h_n}(t) \weakto u(t)$ weakly in $H^\alpha(\R)$ for a.\,e.~$t \in [0,T]$. Furthermore, we now claim that
\begin{equation} \label{eq:strconv}
\mbox{$p_h \Ldif^{J}_h w_h \to c (-\Delta)^{\alpha} w$ strongly in $L^2(\R)$ as $h \to 0^+$.}
\end{equation}
To show \eqref{eq:strconv}, we first recall from \eqref{eq:Lcommute} that 
$$(\Ldif^{J}_h w_h)(x_m) = (\Ldif^{J}_h w)_h(x_m),$$ 
where the action of $\Ldif^{J}_h$ on the function $w=w(x)$ is given by
$$
(\Ldif^{J}_h w)(x) = \frac{1}{\beta(h)} \sum_{n \neq 0} J_{|n|} ( w(x) - w(x-nh) ) .
$$
Applying Lemma \ref{lem:conv}, we obtain that
\begin{equation} \label{eq:strconv2}
\| \Ldif^{J}_h w - c (-\Delta)^{\alpha} w \|_2 \to 0 \quad \mbox{as} \quad h \to  0^+,
\end{equation}
where $c > 0$ is some constant. Using the bound $\| p_h f_h \|_2 \leq C \| f \|_2$ by Lemma \ref{lem:interbound}, we conclude
\begin{align*}
 \| p_h \Ldif^{J}_h w_h - c (-\Delta)^{\alpha} w  \|_2 & = \| p_h (\Ldif^{J}_h w)_h - c (-\Delta)^\alpha w \|_2  \\
 & \leq  \| p_h (  \Ldif^{J}_h w - c (-\Delta)^\alpha w)_h \|_2 +  \| p_h ( c (-\Delta)^\alpha w)_h - c (-\Delta)^\alpha w \|_2 \\
 & \leq C \|  \Ldif^{J}_h w - c (-\Delta)^\alpha w \|_2 + o(1),
\end{align*}
where $o(1) \to 0$ as $h \to 0^+$, thanks to the fact that $\| p_h f_h - f \|_2 \to 0$ as $h \to 0^+$ for any $f \in L^2(\R)$ by Lemma \ref{lem:convph}. (Note that $(-\Delta)^{\alpha} w \in L^2(\R)$ since $w \in C^\infty_0(\R)$.) Using now \eqref{eq:strconv2}, we deduce that \eqref{eq:strconv} holds. 

In view of \eqref{eq:strconv} and by the dominated convergence theorem for the integral with respect to $t$, we find that
$$
\int_0^T \left \langle f p_{h_n} w_{h_n}, p_{h_n} \Ldif^{J}_{h_n} u_{h_n}(t) \right \rangle \, dt \to \int_0^T \left \langle f c (-\Delta)^{\alpha} w,  u(t) \right \rangle \, dt \quad \mbox{as} \quad n \to \infty,
$$
for every $f = f(t) \in C^\infty_0([0,T])$ and $w = w(x) \in C^\infty_0(\R)$. By density this extends to $w \in H^\alpha(\R)$. Since $u(t) \in H^\alpha(\R)$ for a.\,e.~$t \in [0,T]$, we can integrate by parts to conclude that \eqref{eq:weakcv} holds, and hence the claim \eqref{eq:weakconvmaster} follows.

It remains to show weak-$*$ convergence for the nonlinear part. That is, we have to show 
\begin{equation} \label{eq:weakcv2}
\pm \int_0^T \left \langle \Phi, p_{h_n} ( |u_{h_n}(t)|^2 u_{h_n}(t)) \right \rangle \, dt \to \pm \int_0^T \left \langle \Phi,  |u(t)|^2 u(t) \right \rangle \, dt \quad \mbox{as} \quad n \to \infty,
\end{equation} 
for every $\Phi \in L^1([0,T]; H^\alpha(\R))$. Again, by a density argument, it suffices to show this claim for $\Phi(t,x) = f(t) w(x)$ with $f \in C^\infty_0([0,T])$ and $w \in C^\infty_0(\R)$. 

Next, we note that 
$$
\left \| p_h ( |u_h(t)|^2 u_h(t) ) \right \|_2 \leq C \| |u_h(t)|^2 u_h(t) \|_{L^2_h} \leq C \| u_h(t) \|_{L^\infty_h}^2 \| u_h(t) \|_{L^2_h} \leq C,
$$
using Lemma \ref{lem:interbound} and the fact that $\| u_h(t) \|_{L^\infty_h} \leq C$ by Lemma \ref{lem:sobo} and \ref{lem:bounds}. In particular, we can assume that $p_h( |u_h(t)|^2 u_h(t) )$ converges weakly in $L^2(\R)$ for a.\,e.~$t \in [0,T]$. However, from \cite{L}, we recall that $p_h f_h \weakto f$ weakly in $L^2(\R)$ if and only if $q_h f_h \weakto f$ weakly in $L^2(\R)$, where the piecewise constant interpolation $q_h f_h$ was defined in \eqref{def:qh}. Thus it remains to show that
\begin{equation} \label{cv:nonl}
\langle w, q_{h_n}(|u_{h_n}(t)|^2 u_{h_n}(t)) \rangle \to \langle w, |u(t)|^2 u(t) \rangle \quad \mbox{as} \quad n \to \infty,
\end{equation}
for every $w \in C^\infty_0(\R)$ and for a.\,e.~$t \in [0,T]$. 

Indeed, from \eqref{eq:weak1} and by local Rellich compactness, we can assume that $p_{h_n} u_{h_n}(t) \to u(t)$ strongly in $L^2_{\mathrm{loc}}(\R)$ for a.\,e.~$t \in [0,T]$. Next, from \cite{L}, we recall the general fact that $p_h f_h \to f$ strongly in $L^2_{\mathrm{loc}}(\R)$ if and only if $q_h f_h \to f$ strongly in $L^2_{\mathrm{loc}}(\R)$. Since we clearly have that $q_{h_n} ( |u_{h_n}(t)|^2 u_{h_n}(t) ) = |q_{h_n}(u_{h_n}(t))|^2 q_{h_n}(u_{h_n}(t))$ and using the uniform bound $\| u_h(t) \|_{L^\infty_h} \leq C$, we can use the dominated convergence theorem to deduce that \eqref{cv:nonl} holds. This completes the proof of claim \eqref{eq:weakcv2} above.

\medskip
We are now ready to complete the proof of Theorem \ref{thm:main}. From the previous discussion we know that the limit $u \in L^\infty([0,T]; H^\alpha(\R)) \cap W^{1,\infty}([0,T]; H^{-\alpha}(\R))$ given in \eqref{eq:weak1} and \eqref{eq:weak2} satisfies
$$
\int_0^T \langle \Psi, i \partial_t u \rangle \, dt = \int_0^T \langle \Psi, c (-\Delta)^{\alpha} u \rangle \, dt \pm \int_0^T \langle \Psi, |u|^2 u \rangle \, dt ,
$$
for every $\Psi \in L^1([0,T]; H^\alpha(\R))$. In particular, we deduce that \eqref{eq:ivp_ae_t} holds. This completes the proof that the limit $u=u(t,x)$ solves the initial-value problem \eqref{eq:limitivp}. By Proposition \ref{prop:wp2}, the solution $u=u(t,x)$ is unique and satisfies $u \in C^0([0,T]; H^\alpha(\R))$. In particular, the limit $u=u(t,x)$ is independent of the chosen subsequence $h_n \to 0$. The proof of Theorem 1 is now complete.  \hfill $\blacksquare$

\begin{appendix}

\section{Asymptotics for $\omega(k)$}\label{boundomega}

\begin{lemma} \label{lem:asymp}
Let $J = (J_n)_{n=1}^\infty \in \mathcal{K}_s$ for some $0 < s \leq +\infty$ and suppose $J \not \equiv 0$. Define the function
$$
\omega(k) := \sum_{n=1}^\infty J_n \big [ 1- \cos (nk) \big ].
$$
Then there exists some finite constant $C >0$ such that
$$
\lim_{k \to 0} \frac{\omega(k)}{\delta(k)} = C,
$$
where
$$
\delta(k) = \left \{ \begin{array}{ll} |k|^{2s}, & \quad \mbox{if $0 < s < 1$}, \\
(- \log|k|)|k|^2, & \quad \mbox{if $s=1$}, \\
|k|^2, & \quad \mbox{if $1 < s \leq +\infty$}. \end{array} \right . 
$$
\end{lemma}

\begin{remark} \em
In the case when $J_n = n^{-1-2s}$, this result could be inferred from known expansions of $\omega(k)$ in terms of the {\em polylogarithm}. Below, we give a proof that rests on more elementary arguments. 
\end{remark}

\begin{proof}
By symmetry, it suffices to study the limit as $k \to 0^+$. We divide the proof into the following steps. First, we treat the special cases, where 
$$J_n= n^{-1-2s},$$
treating the subcases $0 < s < 1$, $s=1$, and $s \geq 1$ separately. Finally, we turn to the general case $J \in \mathcal{K}_s$.

\medskip
{\bf Case $J_n = n^{-1-2s}$ with $0 < s < 1$.} Let $k > 0$ in what follows. For $J_n = n^{-1-2s}$ with $0 < s < 1$, we write $\omega(k)$ as
$$
\omega(k) =  k^{2s} \sum_{n=1}^\infty \frac{k}{ (k n )^{1+2s}} \big [ 1- \cos (n k) \big ] .
$$  
Passing to the limit $k \to 0^+$, we notice that
$$
\lim_{k \to 0^+} \sum_{n=1}^\infty \frac{k}{(k n )^{1+2s}} \big [ 1- \cos (n k) \big ]  = \int_0^\infty \frac{1- \cos x }{x^{1+2s}}  \, dx,
$$
which can be easily deduced from \cite{D}. Integrating by parts and using an integral table, we find 
\begin{equation} \label{eq:Casymp}
\int_0^\infty \frac{1- \cos x }{x^{1+2s}}  \, dx = \frac{1}{2s} \int_0^\infty \frac{\sin x}{x^{2s}} \, dx = \frac{\pi}{4s \Gamma(2s) \sin ( s \pi)} =: C_s ,
\end{equation} 
where we clearly $C_s > 0$ holds. Hence, we conclude that $\lim_{k \to 0^+} k^{-2s} \omega(k) =2 C_s >0$ holds in this case. 

\medskip
{\bf Case $J_n = n^{-2}$.} First, we recall that $\sum_{n=1}^\infty n^{-1} \cos(nk) = - \log( 2 \sin(k/2))$ for $0 < k \leq \pi$ holds. Integrating this identity twice, we obtain
$$
\sum_{n=1}^{\infty} \frac{1}{n^3} \big [ 1- \cos (n k) \big ]  = -\int_0^k \int_0^z \log (2\sin (t/2)) \, dt \, dz, \quad \mbox{for $0 < k \leq  \pi$}.
$$ 
Clearly, we have $\log (2 \sin (t/2)) = \log 2 + \log ( \sin (t/2))$ and moreover $\int_0^k \int_0^z \log 2 \, dt \, dz = \frac{\log 2}{2} k^2 = \mathcal{O}(k^2)$. Hence it remains to consider the integral involving $\log (\sin t/2)$ only. Now we substitute $u =  \sin (t/2)$ and integrate by parts, which yields that
\begin{align*}
\int_0^z \log(\sin (t/2)) \, dt & = \int_0^{\sin ( z/2)}  \log ( u) \frac{2 du}{\sqrt{1-u^2}} \\
% & =   2 \log ( u) \arcsin ( u ) \Big |_{u=0}^{\sin (z/2)} - 2 \int_0^{\sin (z/2)} \frac{1}{u} \arcsin (u) \, du \\
%& = \log (\sin (z/2)) z - 2 \int_0^{\sin (z/2)} \left ( 1 + \mathcal{O}(u^2) \right ) \, du \\
& = \log (\sin (z/2)) z - 2 \int_0^{\sin (z/2)} \frac{1}{u} \arcsin (u) \, du .
\end{align*}
Using the series expansion $\arcsin u = u + \mathcal{O}(u^3)$, we find that
\begin{align*}
\int_0^k \int_0^z \log(\sin (t/2)) \, dt \, dz = \int_0^k \log (\sin (z/2)) z \, dz + \mathcal{O}(z^2) .
\end{align*}
Next, we integrate by parts again in the integral over $z$ to conclude that
\begin{align*}
\int_0^k \log( \sin (z/2)) z \, dz & = \frac{1}{2} \log ( \sin (k/2)) k^2 - \frac{1}{4} \int_0^k \cot (z/2) z^2 \, dk  \\
& =  \frac{1}{2} \log ( \sin (k/2)) k^2 + \mathcal{O}(k^2),
\end{align*}
where we used that $\cot (z/2) = 2/z + \mathcal{O}(z)$. Since $\lim_{k \to 0^+} \frac{ \log( \sin (k/2))}{\log(k)} = 1$, we conclude that
$$
\lim_{k \to 0^+} \frac{\omega(k)}{\log( k ) k^2} = -\frac{1}{2} .
$$
This completes the proof of Lemma \ref{lem:asymp} for $J_n = n^{-2}$.

\medskip
{\bf Case $J_n = n^{-1-2s}$ with $s > 1$.} Since $\sum_{n=1}^\infty n^{1-2s}$ is finite in this case, we deduce that $\omega'(k)$ and $\omega''(k)$ both exist and are given by
$$
\omega'(k) = \sum_{n=1}^\infty \frac{\sin (nk)}{n^{2s}}, \quad \omega''(k) = \sum_{n=1}^\infty \frac{ \cos (nk) }{n^{2s-1}} .
$$
Note that $\omega(0) = 0$ and $\omega'(0) = 0$. By l'Hospital's rule, we find that
$$
\lim_{k \to 0^+} \frac{\omega(k)}{k^2} = \frac{\omega''(0)}{2} = \frac{1}{2} \sum_{n=1}^{\infty} \frac{1}{n^{2s-1}} = \frac{1}{2} \zeta(2s-1) ,
$$
which is finite, since $2s -1 > 1$ by assumption. This proves Lemma \ref{lem:asymp} for $J_n = n^{-1-2s}$ when $s > 1$. 

\medskip
{\bf Case $J=(J_n)_{n=1}^\infty \in \mathcal{K}_s$.} First, we consider the case  such that $0 < s \leq 1$ holds. Let $A = \lim_{n \to \infty} n^{-1-2s} J_n$. Note that $0 < A < +\infty$ since $J \in \mathcal{K}_s$. Let $\eps > 0$ be given. We claim that we can find $k_0 > 0$ such that
\begin{equation} \label{ineq:Jasymp}
-\eps + (A-\eps) C_s  \leq \frac{\omega(k)}{\delta(k)} \leq \eps + (A+\eps) C_s  , \quad \mbox{for $0 < k < k_0$},
\end{equation}
where $C_s > 0$ is the constant in \eqref{eq:Casymp} and $\delta(k)$ denotes the function introduced in Lemma \ref{lem:asymp} above. Since $\eps > 0$ can be made arbitrarily small, this estimate would show that $\lim_{k \to 0+} \delta(k)^{-1} \omega(k) = C_sA$, as desired. 

To prove \eqref{ineq:Jasymp}, we note that, since $J_n \in \mathcal{K}_s$ by assumption, there exists an integer $N = N(\eps) \geq 1$ such that
$$
\frac{A-\eps}{n^{1+2s}} \leq J_n \leq \frac{A+\eps}{n^{1+2s}}, \quad \mbox{for $n \geq N$}.
$$
Splitting $\omega(k) = \sum_{n < N} \ldots + \sum_{n \geq N} \ldots$ and using that $J_n \geq 0$ and $1-\cos (nk) \geq 0$, we deduce that
\begin{align*}
\omega(k) & \leq \sum_{n < N} J_n \big [ 1 - \cos(nk) \big ] + (A+\eps) \sum_{n \geq N} \frac{1}{n^{1+2s}} \big [ 1- \cos (n k) \big ] \\
& = \sum_{n < N} \left ( J_n - \frac{A+\eps}{n^{1+2s}} \right )  \big [ 1 - \cos(nk) \big ] + (A+\eps) \sum_{n=1}^\infty  \frac{1}{n^{1+2s}} \big [ 1- \cos (n k) \big ] \\
& =: I(N, \eps,k) + II(\eps, k) . 
\end{align*}
Since $I(N,\eps,k)$ is a sum of finitely many terms, we can expand $\cos (nk)$ to conclude that $I(N,\eps, k) = \mathcal{O}(k^2)$ as $k \to 0^+$. Since moreover $0 < s \leq 1$, we can find $k_0 > 0$ such that
$$
\frac{I(N,\eps,k)}{\delta(k)} \leq \eps , \quad \mbox{for $0 < k < k_0$}.
$$
Moreover, from the previous discussion, we deduce that $\frac{II(\eps,k)}{\delta(k)} \to (A+\eps) C_s$ as $k \to 0^+$, where $C_s > 0$ is given by \eqref{eq:Casymp}. Hence, by choosing $k_0>0$ sufficiently small, we deduce
$$
\frac{\omega(k)}{\delta(k)} \leq \eps + (A+\eps) C_s, \quad \mbox{for $0 < k < k_0$},
$$
which is the claimed upper bound in \eqref{ineq:Jasymp}. The proof of the lower bound follows from analogous arguments using that $J_n \geq \frac{A-\eps}{n^{1+2s}}$ for $n \geq N(\eps)$. 

Thus we have shown that \eqref{ineq:Jasymp} holds for arbitrary $\eps > 0$, and this completes the proof of Lemma \ref{lem:asymp} for $J \in \mathcal{K}_s$ with $0 < s \leq 1$.

Finally, it remains to treat the case $J \in \mathcal{K}_s$ with $1 < s \leq +\infty$. Since $\sum_{n=1}^\infty n^2 J_n < +\infty$ in this case, we can deduce in a similar fashion as for $J_n = n^{-1-2s}$ with $s > 1$ that
$$
\lim_{k \to 0} \frac{\omega(k)}{|k|^2}= \frac{\omega''(0)}{2} = \frac{1}{2} \sum_{n=1}^\infty n^2 J_n < +\infty .
$$ 
The proof of Lemma \ref{lem:asymp} is now complete. \end{proof}

\section{Dual Bounds and Integration by Parts}\label{dual}

Recall the definition of $\| \cdot \|_{H^\sigma_h}$ in \eqref{def:Hh_sigma} with $0 \leq \sigma \leq 1$. We define the corresponding dual norm by setting
$$
\| u_h \|_{H^{-\sigma}_h} := \sup_{ \| v_h \|_{H^\sigma_h} \leq 1} \left | (v_h, u_h)_{L^2_h} \right | .
$$
We have the following fact.

\begin{prop} \label{prop:dual_ph}
For any $0 \leq \sigma \leq 1$, we have
$$
\| p_h f_h \|_{H^{-\sigma}} \leq C \| f_h \|_{H^{-\sigma}_h}
$$
with some constant $C > 0$ independent of $h> 0$ and $f_h$.
\end{prop}

\begin{proof}
This claim is easily verified for $\sigma=0$ (see Lemma \ref{prop:unifH2}) and $\sigma=1$ (by calculation using also the equivalence of the norms $\| \cdot \|_{\tilde{H}^1_h}$ and $\| \cdot \|_{H^1_h}$ by Lemma \ref{lem:hfull}). By interpolation of norms (as in the proof of Lemma \ref{prop:unifH2}) we deduce the bound for all $0 \leq \sigma \leq 1$. 
\end{proof}

Next, we have the following estimate.

\begin{prop} \label{prop:dualbound}
Suppose that $J= ( J_n )_{n=1}^\infty$ satisfies (A1) and (A2) with some $0 < s \leq +\infty$. Let $\alpha = s$ for $0 < s < 1$ and $\alpha = 1$ for $s \geq 1$. Then there exists a constant $C > 0$ independent of $h >0$ such that
$$
\| \Ldif^{J}_h u_h \|_{H^{-\alpha}_h} \leq C \| u_h \|_{H^\alpha_h} 
$$
for all $u_h \in L^2_h$.
\end{prop}

\begin{proof}
This follows from estimates derived in the proofs of Lemma \ref{lem:embed} and \ref{lem:embedlog}. Indeed, with the notation used there, we find that
\begin{align*}
\left | (v_h, \Ldif^{J}_h u_h)_{L^2_h} \right | & \leq h \int_{-\pi}^{+\pi} \left (1+ \frac{ \omega(k) }{\beta(h)} \right )  \left | \hat{v}_h(k) \right | \left | \hat{u}_h(k) \right | \, dk \\
& \leq C h \int_{-\pi}^{\pi} \left ( 1 + h^{-2\alpha} |k|^{2 \alpha} \right ) \left | \hat{v}_h(k) \right | \left | \hat{u}_h(k) \right | \, dk \leq C \| v_h \|_{H^\alpha_h} \| u_h \|_{H^{\alpha}_h} ,
\end{align*}
with some constant $C > 0$ independent of $h> 0$. Here we used the estimates for $(1+ \frac{\omega(k)}{\beta(h)})$ derived in the proofs of Lemma \ref{lem:embed} and \ref{lem:embedlog}. \end{proof}

We have the following technical result.
\begin{lemma} \label{lem:discintpart}
Let $\Ldif^{J}_h$ be as above. For any $w_h, u_h \in L^2_h$, we have the identity
$$
\left \langle p_h w_h, p_h \Ldif^{J}_h u_h \right \rangle = \left \langle p_h \Ldif^{J}_h w_h, p_h u_h \right \rangle,
$$
where $\langle f,g \rangle = \int_{\R} \overline{f(x)} g(x) \,dx$ is the usual inner product on $L^2(\R)$.
\end{lemma}

\begin{proof}
First, we recall that (with $x=mh$ and $m \in \Z$)
$$
(p_h w_h)(x) = \sum_{m} w_h(x_m) \1_{[x_m, x_{m+1})}(x) + \sum_m (D^+_h w_h)(x_m) \1_{[x_m, x_{m+1})}(x) (x-x_m),
$$
where $\1_A(x)$ denotes the characteristic function of the set $A \subset \R$. Using this, we conclude that
\begin{align*}
\left \langle p_h w_h, p_h \Ldif^{J}_h u_h \right \rangle & =  h \sum_m \overline{w_h(x_m)} ( \Ldif^{J}_h u_h)(x_m) \\
& \quad + \frac{1}{2} h^2 \sum_m   \overline{(D^+_h w_h)(x_m)} ( \Ldif^{J}_h u_h)(x_m) \\
& \quad + \frac{1}{2} h^2 \sum_m   \overline{w_h(x_m)} (D^+_h \Ldif^{J}_h u_h)(x_m)\\
& \quad + \frac{1}{3} h^3 \sum_m  \overline{(D^+_h w_h)(x_m)} ( D^+_h \Ldif^{J}_h u_h)(x_m) . \\
\end{align*}
Since $(\Ldif^{J}_h)^{*} = \Ldif^{J}_h$ is self-adjoint on $L^2_h$ and using the commutation relation $D^+_h \Ldif^{J}_h = \Ldif^{J}_h D^+_h$, we easily  derive the claimed identity.
\end{proof}

\end{appendix}

\bibliographystyle{amsalpha} 
%\thebibliography{hh}

\thispagestyle{plain}

\end{document}